\documentclass[reqno]{amsart}
\usepackage{amssymb,amscd,verbatim, amsthm,graphics, color,latexsym,amsmath,multicol,mathtools}
\usepackage{fancybox,lscape,multirow}
\usepackage{longtable}

\usepackage[all]{xy}

\usepackage[top=2.7cm, bottom=2.7cm, left=2.7cm, right=2.7cm]{geometry}

\newcommand \C[1]{{\mathcal #1}}

\newcommand \wti[1]{{\widetilde {#1}}}

\newcommand\fg{\mathfrak g}

\newcommand \bC{{\mathbb C}}
\newcommand \bF{{\mathbb F}}

\newcommand \bO{{\mathbb O}}

\newcommand\CG{{\C G}}

\newcommand\CO{{\C O}}
\newcommand\CL{{\C L}}

\newcommand\CU{{\C U}}

\newcommand\fh{{\mathfrak h}}

\newcommand\sk{{\mathsf k}}

\newcommand\AZ{{\mathsf{AZ}}}
\newcommand\GL{{\mathsf{GL}}}

\newcommand\bfG{{\mathbf G}}
\newcommand\rG{{\mathrm G}}

\newtheorem{theorem}{Theorem}[section]
\newtheorem{conjecture}[theorem]{Conjecture}
\newtheorem{corollary}[theorem]{Corollary}
\newtheorem{lemma}[theorem]{Lemma}
\newtheorem{proposition}[theorem]{Proposition}
\newtheorem{definition}[theorem]{Definition}
\newtheorem{remark}[theorem]{Remark}
\newtheorem{example}[theorem]{Example}

\newcommand\Ind{\operatorname{Ind}}

\newcommand\St{\mathsf{St}}

\newcommand\Irr{\mathsf{Irr}}

\newcommand\res{\mathsf{res}}

\newcommand\ind{\mathsf{ind}}

\newcommand\Fr{\mathsf{Fr}}
\newcommand\Ft{\mathsf{FT}}
\newcommand\WF{\mathsf{WF}}

\newcommand\FT{\operatorname{FT}}

\def\<{\langle} 
\def\>{\rangle}

\def\Ureg{{\mathrm G_2}}
\def\Usreg{{\mathrm G_2(a_1)}}
\def\Usr{{\widetilde A_1}}
\def\Ulr{{A_1}}
\def\Utr{{1}}

\numberwithin{equation}{section}

\begin{document}

\title[The Wavefront Set: Bounds]{The Wavefront Set: Bounds for the Langlands parameter}

\author
{Dan Ciubotaru}
        \address[D. Ciubotaru]{Mathematical Institute, University of Oxford, Oxford OX2 6GG, UK}
        \email{dan.ciubotaru@maths.ox.ac.uk}

\author
{Ju-Lee Kim}
        \address[J.-L. Kim]{Department of Mathematics, M.I.T., Cambridge MA 02139, USA}
        \email{juleekim@mit.edu}

\begin{abstract}{For an irreducible smooth representation of a connected reductive $p$-adic group, two important associated invariants are the wavefront set and the (partly conjectural) Langlands parameter. While a wavefront set consists of $p$-adic nilpotent orbits, one constituent of the Langlands parameter is a complex nilpotent orbit in the dual Lie algebra. For unipotent representations in the sense of Lusztig, the corresponding nilpotent orbits on the two sides are related via the Lusztig--Spaltenstein duality, \cite{CMBO21,CMBO23}. In this paper, we formulate a general upper-bound conjecture and several variants relating the nilpotent orbits that appear in  the wavefront set and in the Langlands parameter. We also verify these expectations in some cases, including the depth-zero supercuspidal representations of classical groups and all the irreducible representations of $G_2$.}
\end{abstract}


\maketitle

\section{\bf Introduction: a conjecture}

Let $\sk$ be a nonarchimedean local field of characteristic $0$ with residue field $\mathbb F_q$. Let $\bar\sk$ be a separable closure of $\sk$ and let $K\subset \bar\sk$ denote the maximal unramified extension. Let $W_\sk$ be the Weil group with inertia subgroup $I_\sk$ and norm $||~||$. Let $W'_\sk$ be the Weil-Deligne group, i.e., $W'_\sk=W_\sk\ltimes \bC$, where the action is defined by $w\cdot z=||w|| z$, $w\in W_\sk$, $z\in \bC$.
Let $\Fr$ denote the geometric Frobenius element, which generates $W_\sk/I_\sk$.

Let $\bfG$ be a connected reductive $\sk$-group that splits over an unramified extension. The {\it pure inner forms} of $\bfG$, in the sense of \cite[Definition 2.6]{vogan-llc} are parametrized by the first Galois cohomology group $H^1(\sk,\bfG)$. For every $\omega\in H^1(\sk,\bfG)$, let $\bfG^\omega(\sk)$ denote the corresponding pure inner form. Let $G^\vee$ be the complex dual group to $\bfG$. Let ${}^L\!G=G^\vee\rtimes W_\sk$ be the $L$-group.

\subsection{Langlands parameters} We briefly recall the conjectural Langlands parametrization of the irreducible $\bfG^\omega(\sk)$-representations, in the formulation of \cite[\S4]{vogan-llc}, generalizing the Kazhdan--Lusztig geometric parametrization of representations with Iwahori-fixed vectors \cite{KL}. A \emph{Langlands parameter} is a group homomorphism
\begin{equation}
\varphi: W'_\sk\to  {}^L\!G
\end{equation}
satisfying certain compatibility conditions. Let $A_\varphi$ denote the group of components of the centralizer in $G^\vee$ of the image of $\varphi$. One expects a one-to-one correspondence
\[ \sqcup_{\omega\in H^1(\sk,\bfG)}\Irr \,\bfG^\omega(\sk)\leftrightarrow G^\vee\text{-conjugacy classes of }(\varphi,\phi),
\]
where $\varphi$ is a Langlands parameter and $\phi$ is an irreducible representation of $A_\varphi$. The pairs $(\varphi,\phi)$ are called \emph{enhanced} Langlands parameters. Let $\Pi(\varphi)$ denote the Langlands packet, i.e., the set of representations parametrized by a fixed ($G^\vee$-orbit of) $\varphi$. We need a geometric reformulation of this correspondence.

Fix a group homomorphism $\lambda: I_\sk\to  {}^L\!G$ with finite image. The set of enhanced Langlands parameters for which $\varphi|_{I_\sk}=\lambda$ admits the following interpretation. Let $H^\vee={G^\vee}^{\lambda(I_\sk)}$, a possibly disconnected reductive group, and $\fh^\vee={\fg^\vee}^{\lambda(I_\sk)}$, a reductive subalgebra of $\fg^\vee$, denote the fixed points.  

Let $s$ be $\varphi(\Fr)\in {}^L\!G$, a semisimple element. Since $\Fr$ normalizes $I_\sk$, it follows that $s$ normalizes $H^\vee$ (and $\fh^\vee$). Let $\fh^\vee_q$ denote the $q$-eigenspace of $\mathrm{Ad}(s)$. Let $H^\vee(s)$ be the centralizer of $s$ in $H^\vee$. Then $H^\vee(s)$ acts on $\fh^\vee_q$ with finitely many orbits. There is a one-to-one correspondence between $G^\vee$-orbits of enhanced Langlands parameters $(\varphi,\phi)$ such that $\varphi(\Fr)=s$ and $\varphi|_{I_\sk}=\lambda$, and pairs $(\CO^\vee, \mathcal L^\vee)$ where
\begin{itemize}
\item $\CO^\vee$ is an $H^\vee(s)$-orbit in $\fh^\vee_q$;
\item $\mathcal L^\vee$ is an irreducible $H^\vee(s)$-equivariant local system supported on $\CO^\vee$.
\end{itemize}
Denote by $\Pi(\lambda,s,\CO^\vee)=\Pi(\varphi)$ the corresponding $L$-packet, each of whose member $\pi(\lambda,s,\CO^\vee,\CL^\vee)$ should be parametrized by the equivariant local systems $\mathcal L^\vee$ on $\CO^\vee$. 
A desideratum is that if $\Pi(\lambda,s,\CO^\vee)$ consists of tempered representations then $\CO^\vee$ is dense in $\fh^\vee_q$.

\subsection{Aubert--Zelevinsky involution} Let $\AZ$ denote the normalized Aubert--Zelevinsky involution \cite[Definition 1.5, Corollary 3.9]{Au-inv} on $\Irr\, \bfG^\omega(\sk)$ for every form $\bfG^\omega(\sk)$. If $\varphi_\pi$ is the Langlands parameter of $\pi$, then it is expected that $\varphi_\pi(W_\sk)$ and $\varphi_{\AZ(\pi)}(W_\sk)$ are $G^\vee$ conjugate. Assuming this, denote the enhanced Langlands parameter of $\AZ(\pi)$ as follows
\begin{equation}
\AZ(\pi(\lambda,s,\CO^\vee,\CL^\vee))=\pi(\lambda,s,\CO^\vee_{\AZ},\CL^\vee_{\AZ}).
\end{equation}
It is expected that the {enhanced} Langlands parameter of $\AZ(\pi)$ is related to that of $\pi$ via the Fourier transform \cite[\S3.7]{KS} on the appropriate space of perverse sheaves of Langlands parameters, see also \cite[\S3]{evens-mirkovic}. This is known to hold for all representations with Iwahori-fixed vectors by \cite[Theorem 0.1]{evens-mirkovic}.

More precisely, suppose that $\pi$ has Iwahori-fixed vectors, so $\lambda(I_\sk)=1$ and that $G^\vee$ is simply connected. The geometric space of parameters consists of irreducible $G^\vee(s)$-equivariant perverse sheaves on the complex vector space $\fg^\vee_q$. The connected reductive group $G^\vee(s)$ acts with finitely many orbits on $\fg^\vee_q$. This implies that $G^\vee(s)$ also acts with finitely many orbits on the dual vector space $(\fg^\vee_q)^*$. Identify $(\fg^\vee_q)^*$ with $\fg^\vee_{q^{-1}}$ by using a fixed nondegenerate invariant symmetric form on $\fg^\vee$. Pyasetsky's duality \cite{zelevinsky-p-adic} says that there exists a one-to-one correspondence between the orbits in $\fg^\vee_q$ and in $\fg^\vee_{q^{-1}}$, $\CO^\vee\to P(\CO^\vee)$, given as follows: for each $\CO^\vee\subset \fg^\vee_q$, $P(\CO^\vee)\subset \fg^\vee_{q^{-1}}$ is the unique orbit such that the conormal bundles have the same closures 
\[
\overline {T_{\CO^\vee} \fg^\vee_q}=\overline {T_{P(\CO^\vee)} \fg^\vee_{q^{-1}}}
\]
in $ \fg^\vee_q\times  \fg^\vee_{q^{-1}}$.

Let $\FT$ denote the Fourier transform from $G^\vee(s)$-equivariant perverse sheaves on $\fg^\vee_q$ to $G^\vee(s)$-equivariant perverse sheaves on $\fg^\vee_{q^{-1}}$, \cite[\S3.2]{evens-mirkovic}. Let $\mathsf{IC}(\CO,\CL)$ denote the simple perverse sheaf given by the intersection cohomology complex defined by the equivariant irreducible local system $\CL$ supported on $\CO$. Then $\FT(\mathsf{IC}(\CO^\vee,\CL^\vee))=\mathsf{IC}(\CO^\vee_\Ft,\CL^\vee_\Ft)$ for an orbit $\CO^\vee_\Ft$ in $\fg^\vee_{q^{-1}}$ and an equivariant local system $\CL^\vee_\Ft$. We emphasize that $\CO^\vee_\Ft$ depends on $\CL^\vee$ as well, not just on $\CO^\vee$. If $\fg^\vee=\mathfrak{gl}(n)$, so that all local systems that appear are trivial, \cite[Proposition 7.2]{evens-mirkovic} shows that
\[
\CO^\vee_\Ft=P(\CO^\vee).
\]
In general however, this is not the case, as one can see already for $G^\vee=Sp(4)$ (\cite[Remark 7.5]{evens-mirkovic}). We identify, as we may, the $G^\vee(s)$-orbits on $\fg^\vee_{q^{-1}}$ with the $G^\vee(s^{-1})$-orbits on $\fg^\vee_q$. Then the main result (Theorem 0.1) in {\it loc. cit.} is that  
\begin{equation}\label{e:AZ-geom}
\AZ(\pi(1,s,\CO^\vee,\CL^\vee))^*=\pi(1,s^{-1},\CO^\vee_\Ft,\CL^\vee_\Ft),
\end{equation}
where $*$ in the left hand side denotes the contragredient. In terms of the previous notation, this means, in particular, that for Iwahori-spherical representations, $G^\vee\cdot\CO^\vee_{\AZ}=G^\vee\cdot\CO^\vee_{\Ft}$. In our conjectures below, we will be interested in the nilpotent part of the Langlands parameter of $\AZ(\pi)$, that is, in $\mathbb O^\vee_{\AZ}=G^\vee\cdot \CO^\vee_\AZ.$

\subsection{Spaltenstein and Lusztig-Sommers dualities}\label{s:dualities} Let $G$ denote a complex group with maximal torus $T$ with dual $G^\vee$ and dual maximal torus $T^\vee$. Let $W$ denote the Weyl group of $(G,T)$ and $W^\vee$ the Weyl group of $(G^\vee,T^\vee)$. Of course, $W\cong W^\vee$. Let $\CU_G$ and $\CU_{G^\vee}$ denote the sets of unipotent elements in $G$ and $G^\vee$ respectively. Let $G\backslash\CU_G$, $G^\vee\backslash\CU_{G^\vee}$ denote the corresponding sets of conjugacy classes.

Let $d: G\backslash \CU_G\to G^\vee\backslash\CU_{G^\vee}$ and $d^\vee: G^\vee\backslash \CU_{G^\vee}\to G\backslash\CU_{G}$ denote the duality maps defined by Spaltenstein \cite[\S10]{spaltenstein}. The images of $d,d^\vee$ consist of the set of special unipotent conjugacy classes.

Let $\fg$ be the Lie algebra of $G$. We will also make use of the duality map $d_S:G\backslash G\to G^\vee\backslash \CU_{G^\vee}$ from conjugacy classes of $G$ to unipotent conjugacy classes of $G^\vee$, defined by Lusztig \cite[\S13.3]{orange} in the case of special conjugacy classes and generalized by Sommers \cite{sommers-duality}. Suppose $x\in G$ has a Jordan decomposition $x=su$ and assume $s\in T$. Let $W(s)$ denote the centralizer of $s$ in $W$, this is  the Weyl group of the centralizer $G(s)$. Let $E^{W(s)}_{u,1}$ denote the Springer $W(s)$-representation attached to the trivial local system for the unipotent class of $u$ in $G(s)$. Consider the induced representation $\Ind_{W(s)}^W(E^{W(s)}_{u,1})$. There exists a unique irreducible $W$-representation $\hat E$ in this induced representation such that $b_{\hat E}=b_{E^{W(s)}_{u,1}}$, where $b$ denotes the lowest harmonic degree. Regard $\hat E$ as an irreducible representation of $W^\vee$. By Springer's correspondence for $G^\vee$, there exists a unique unipotent conjugacy class $u^\vee$ in $G^\vee$ such that $\hat E$ is realized in the top cohomology of the Springer fiber for $u^\vee$. Define $d_S(x)=u^\vee$. We recall that
\[ d_S|_{G\backslash \CU_G}=d.
\]

We will often identify the unipotent classes in $G^\vee$ with the nilpotent orbits in $\fg^\vee$ and think of the image of the map $d_S$ as lying in the nilpotent cone of $\fg^\vee$.

One can flip the roles of $G$ and $G^\vee$ and have the dual map $d^\vee_S:G^\vee\to G$.

\subsection{Unramified wavefront set} If $\pi$ is an admissible $\bfG^\omega(\sk)$-representation, let ${}^K\WF(\pi)$ denote the unramified wavefront set of $\pi$ as in \cite{okada-WF}. This is a (finite) collection of nilpotent $\bfG^\omega(K)$-orbits in $\mathbf{\mathfrak g}^\omega(K)$. Using the parametrization in {\it loc. cit.}, there is a one-to-one correspondence between the set of nilpotent $\bfG^\omega(K)$-orbits and pairs
\[ (\CO,C):\ \CO \text{ nilpotent $\bfG(\bar \sk)$-orbit in }\fg(\bar\sk),\ C\text{ conjugacy class in }A(\CO),
\]
where $A(\CO)$ is a group of components of the centralizer $C_{\bfG(\bar\sk)}(u)$ in $\bfG(\bar\sk)$ of a representative $u$ of $\CO$. By abuse of notation, we think of ${}^K\WF(\pi)$ as a collection of such pairs $(\CO,C)$. For each $\CO$ choose a representative $u\in \CO$. For each $C$, we may choose a semisimple representative $\tau$ in $C_{\bfG(\bar\sk)}(u)$. Choosing an isomorphism $\bar \sk\cong \mathbb C$, we may identify $\bfG(\bar\sk)$ with the complex group $G=\bfG(\mathbb C)$ with the same root datum, and think of the duality maps in section \ref{s:dualities} as maps between $\bfG(\bar\sk)$ and $G^\vee$. Write
\[d_S(\CO,C)=d_S(\tau u),
\]
which is independent of the choice of $u$ and $\tau$. 
Also write
\[d_S({}^K\WF(\pi))=\bigcup_{(\CO,C)\in {}^K\WF(\pi)} d_S(\CO,C),
\]
a union of nilpotent orbits in $\fg^\vee$. Define $d({}^K\WF(\pi))$ similarly.

Also in \cite{okada-WF}, the notion of {\it canonical unramified wavefront set} ${}^K\!\underline\WF(\pi)$ was introduced. This is a collection of classes under a certain equivalence in ${}^K\WF(\pi)$. The invariant ${}^K\!\underline\WF(\pi)$ is a set of pairs 
\[ (\CO,\bar C):\ \CO \text{ nilpotent $G$-orbit in }\fg,\ \bar C\text{ conjugacy class in }\bar A(\CO),
\]
where $\bar A(\CO)$ is Lusztig's canonical quotient. The duality $d_S$ factors through the canonical quotient, so we may talk about $d_S({}^K\!\underline\WF(\pi))$.

\subsection{Conjectures} Let ${}^{\bar \sk}\WF(\pi)$ denote the geometric wavefront set of an admissible representation of $\bfG^\omega(\sk)$. This is a collection of adjoint nilpotent orbits in $\fg(\bar\sk)$. The nilpotent orbits in $\fg(\bar\sk)$ are ordered with respect to the partial order given by the Zariski closure. The nilpotent orbits in ${}^{\bar \sk}\WF(\pi)$ are by definition non-comparable in the closure order. In almost all known examples, ${}^{\bar \sk}\WF(\pi)$ consists of a single nilpotent orbit if $\pi$ is irreducible, but not always, as recent examples of Tsai \cite{tsai} have shown.

If $\mathcal S_1,\mathcal S_2$ are two sets of nilpotent orbits, we write
\[
\mathcal S_1\le \mathcal S_2,
\]
if for every $\mathcal O_1\in \mathcal S_1$, there exists $\mathcal O_2\in \mathcal S_2$ such that $\mathcal O_1\le \mathcal O_2$ in the closure order.

\begin{conjecture}\label{conj-main}
Let $\pi(\lambda,s,\CO^\vee,\CL^\vee)$ be an irreducible $\bfG^\omega(\sk)$-representation in the $L$-packet $\Pi(\varphi)=\Pi(\lambda,s,\CO^\vee)$. Then
\begin{enumerate}
\item  $ {}^{\bar \sk}\WF(\AZ(\pi))\le d^\vee(G^\vee\cdot \CO^\vee).$ Equivalently,
\item[(1')] $\CO^\vee\subset \overline{d({}^{\bar \sk}\WF(\AZ(\pi)))}$. 
\end{enumerate}
\end{conjecture}

Of course, the conjecture is contingent on the existence of a local Langlands correspondence. 

\begin{remark}While this paper was being finalized, we learned that Conjecture \ref{conj-main}, as well as the connection with Conjecture \ref{conj-main-2} and Jiang's conjecture, have also been formulated independently and studied for classical groups by Hazeltine, Liu, Lo, and Shahidi \cite{HLLS}. In {\it loc. cit.}, the conjectures are stated and investigated for general groups (rather than `inner to split'), see \cite[Conjecture 1.1, Remarks 1.9, 7.11]{HLLS}.
\end{remark}

In terms of the discussion around (\ref{e:AZ-geom}), the conjecture could be rephrased as
\begin{equation}
{}^{\bar \sk}\WF(\pi)\le d^\vee(\mathbb O^\vee_\Ft)=d^\vee(\mathbb O^\vee_{\AZ}),
\end{equation}
whenever we know (or expect) that the action of $\AZ$ on the space of geometric parameter is given by the geometric Fourier transform $\FT$.

\begin{remark} 
The conjecture aims to generalize the recent results for Lusztig unipotent representations (trivial $\varphi(I_\sk)$) with real infinitesimal character, see \cite{CMBO21,CMBO23}, and for unipotent supercuspidal representations \cite{CMBO-cusp}. In these cases, the upper bound is an equality. 

But even in the case of unipotent representations there is an example for the split group $E_7$ \cite[\S1]{CMBO23} that shows that for unipotent representations with nonreal infinitesimal character, the upper bound can be strict. 
At the end of section \ref{s:G2}, we will comment on other possible versions of the conjecture involving the duality $d_S$ and ${}^K\!\underline\WF(\pi)$.
\end{remark}

For $\GL(n)$, Conjecture \ref{conj-main} is true, in fact a stronger statement holds. 
\begin{theorem}[{M\oe glin-Waldspurger \cite[II.2]{MW}}] For every irreducible $\GL(n,\sk)$-representation $\pi$, {${}^{\bar \sk}\WF(\AZ(\pi))$ is a singleton and} 
\begin{equation}\label{e:gl}
    {}^{\bar \sk}\WF(\AZ(\pi))=\{d^\vee(\mathbb O_\pi^\vee)\},
    \end{equation}
    where $\bO_\pi^\vee$ is the $G^\vee$-saturation of the nilpotent Langlands parameter of $\pi$.
\end{theorem}

To exemplify Conjecture \ref{conj-main}, we will verify it for the depth-zero supercuspidal representations of classical groups, using the Langlands parametrization of Lust--Stevens \cite{LS}, and for the irreducible representations of $G_2$, using Aubert--Xu \cite{AX22} and Gan--Savin \cite{GS1,GS23}.

\begin{remark}
If {an irreducible smooth representation $\pi$} is generic, then its geometric wavefront set is the principal orbit. For the conjecture to hold, we would need $\mathbb O^\vee_\AZ=0$, so $\CO^\vee$ would need to be the open dense orbit in $\fh^\vee_q$. Moreover, 
$\CL^\vee$ would have to be the trivial local system.
\end{remark}

In the case of supercuspidal representations, Conjecture \ref{conj-main} takes a simpler, more concrete form. If $\pi$ is supercuspidal, the Langlands parameter is expected to satisfy
\begin{enumerate}
\item $\CO^\vee=\CO_{\lambda,s}^\vee$, the unique open dense $H^\vee(s)$-orbit in $\fh^\vee_q$ (this is in fact the case whenever the representation $\pi$ is tempered).
In addition,
\item $\AZ(\pi)=\pi$ (immediate from the definition of $\AZ$).
\end{enumerate}

\begin{conjecture}[Supercuspidals]\label{conj-cusp} Let $\pi(\lambda,s,\CO^\vee_{\lambda,s},\CL^\vee)$ be an irreducible supercuspidal $\bfG^\omega(\sk)$-representation in the $L$-packet $\Pi(\varphi)=\Pi(\lambda,s,\CO^\vee_{\lambda,s})$. Set $\bO^\vee_{\varphi}=G^\vee\cdot \CO^\vee_{\lambda,s}$. Then
\begin{enumerate}
\item \[ {}^{\bar \sk}\WF(\pi)\le d^\vee(\bO^\vee_{\varphi}).
\]
Equivalently, $\bO^\vee_{\varphi}\le d({}^{\bar \sk}\WF(\pi))$. 
\item If in addition, $\bfG^\omega(\sk)$ is split and $\pi$ is depth-zero compactly induced from a cuspidal representation of the maximal hyperspecial parahoric subgroup, then 
\[ {}^{\bar \sk}\WF(\pi)=\{d^\vee(\bO^\vee_{\varphi})\}.
\]
\end{enumerate}
\end{conjecture}
The first part of the conjecture is just a particular case of Conjecture \ref{conj-main}.

\begin{remark}
    Notice that Conjecture \ref{conj-cusp} is vacuously true for the nonsingular supercuspidal representations, in the sense of Kaletha \cite{Kal}, because $\bO^\vee_\varphi=0$ so $d^\vee(\bO^\vee_\varphi)$ is the principal geometric orbit. In this case, the upper bound is attained only for the generic constituents of the $L$-packets. In section \ref{s:DBR}, we will compute the wavefront sets of the regular depth-zero supercuspidals parameterized by Kazhdan-Varshavsky \cite{KV} and DeBacker-Reeder \cite{DBR}.
\end{remark}

In section \ref{s:classical}, we will prove
\begin{theorem}[Depth-zero supercuspidals of classical groups] Suppose $\bfG(\sk)$ is a symplectic, special orthogonal, or unitary group. Conjecture \ref{conj-cusp} holds for all depth-zero supercuspidal representations, using the Langlands parametrization of Lust--Stevens \cite{LS}.
\end{theorem}

We mention that Conjectures \ref{conj-main} and \ref{conj-cusp} are compatible with a conjecture about Arthur packets formulated by D. Jiang. Suppose $\varphi$ is a tempered $L$-parameter. Then $\AZ(\Pi(\varphi))$ is expected to be an Arthur packet. Let $\pi'\in\AZ(\Pi(\varphi))$ be one of these ``anti-tempered'' representations. Then $\pi'=\AZ(\pi(\lambda,s,\CO^\vee_{\lambda,s},\CL^\vee))$ for some local system $\CL^\vee$. Conjecture \ref{conj-main} then says that we should have
\begin{equation}
{}^{\bar \sk}\WF(\pi')\le d^\vee(\mathbb O^\vee_\varphi).
\end{equation}
The orbit $\CO^\vee_{\varphi}$ is precisely the ``Arthur $SL(2)$'' $\bO^\vee_{A,\varphi}$ attached to the Arthur packet $\AZ(\Pi(\varphi))$. In other words, the expectation is
\begin{equation}
{}^{\bar \sk}\WF(\pi')\le d^\vee(\bO^\vee_{A,\varphi}),
\end{equation}
for all $\pi'\in \AZ(\Pi(\varphi))$. This is known to hold (with equality) if $\varphi$ is a ``real'' unramified $L$-parameter, see \cite{CMBO-arthur}. Jiang's conjecture \cite[Conjecture 4.2]{Ji} says that this inequality holds for {\it all} Arthur packets, when they are defined (not just the above anti-tempered packets) and moreover that there exists at least one constituent of the Arthur packet where the equality is achieved. One can naturally reformulate Conjecture \ref{conj-main} in this case as follows:

\begin{conjecture}\label{conj-main-2}
Let $\pi=\pi(\lambda,s,\CO^\vee_{\lambda,s},\CL^\vee)$ be an irreducible representation in the tempered $L$-packet $\Pi(\varphi)=\Pi(\lambda,s,\CO^\vee_{\lambda,s})$. Then
\begin{enumerate}
\item  $ {}^{\bar \sk}\WF(\AZ(\pi))\le d^\vee(\mathbb O^\vee_\varphi).$ 
\item Equality in (1) is achieved for some $\pi$ in the $L$-packet.
\end{enumerate}
\end{conjecture}

In section \ref{sec:induction}, we will explain that Conjecture \ref{conj-main-2}(1) (via the usual desiderata of Langlands correspondence) implies Conjecture \ref{conj-main}.

In section \ref{s:G2}, we will verify Conjecture 1.9 for $G_2$, using the Langlands correspondences from \cite{AX22,GS1,GS23} and the wavefront set results of \cite{JLS,LoSa}:

\begin{theorem}
    Conjectures 1.1 and 1.9 hold for $G_2$.
\end{theorem}

\begin{remark}\label{r:extreme} We look at the two extreme cases of Conjectures \ref{conj-main} and \ref{conj-main-2}. 
\begin{enumerate}
    \item[(a)] 
Firstly, let us call an irreducible representation $\pi$ \emph{anti-generic} if $\AZ(\pi)$ is generic. If $\pi$ is anti-generic, then $ {}^{\bar \sk}\WF(\AZ(\pi))=\mathcal O_{\textup{reg}}$, the regular nilpotent orbit. The only way that the upper bound conjecture could hold then is if the Langlands nilpotent parameter of $\pi$ is $0$. In other words, for every anti-generic representation  the image of $\mathbb C\subset W_\sk'$ under the Langlands parameter is trivial. 

If $\pi$ is unipotent anti-generic, then it is known that it is spherical (with respect to the maximal special parahoric) and the desired condition on the Langlands parameter holds. 
If $\pi$ is supercuspidal, since $\AZ(\pi)=\pi$, the conjecture implies that every generic supercuspidal representation has trivial nilpotent Langlands parameter.

For classical groups, the algorithm for computing the $\AZ$-dual in \cite{AM} implies that the Langlands nilpotent parameter of an anti-generic representation is indeed trivial.
\smallskip

\item[(b)]
Secondly, suppose $\mathcal O^\vee=0$, so $d^\vee(G^\vee\cdot \mathcal O^\vee)=\mathcal O_{\textup{reg}}.$ Then Conjecture \ref{conj-main-2}(1) is trivially satisfied, but (2) says that there should exist a representation $\pi$ in the tempered $L$-packet such that ${}^{\bar \sk}\WF(\AZ(\pi))=\mathcal O_{\textup{reg}}$. In other words, every tempered $L$-packet with trivial Langlands nilpotent parameter should contain an anti-generic representation.

This is equivalent with a conjecture of Shahidi as follows. If $\varphi\leftrightarrow (\lambda,s,\CO^\vee_{\lambda,s})$ is a tempered Langlands parameter, since $\CO^\vee_{\lambda,s}$ must be dense in $\fh^\vee_q$, we can only have $\CO^\vee_{\lambda,s}=0$ if $\fh^\vee_q=0$. But then $\varphi$ is the only $L$-parameter for the pair $(\lambda,s)$ and since it is expected that $\AZ$ preserves the part $(\lambda,s)$ of the Langlands parameter, it follows that
\[\AZ(\Pi(\varphi))=\Pi(\varphi).
\]
\end{enumerate}
\end{remark}

\section{\bf Parabolic induction}\label{sec:induction}

In this section, we discuss the compatibility of our conjectures with parabolic induction. Let $\mathbf P=\mathbf M\mathbf N$ be a $\mathsf k$-rational parabolic subgroup of $\bfG$ with a $\mathsf k$-rational Levi subgroup $M=\mathbf M(\sk)$. 

If $\pi=i_{\mathbf P(\sk)}^{\bfG(\sk)}(\pi_M)$ is a parabolically induced representation, where $\pi_M$ is an admissible representation of $M$, then the relation between the wavefront sets is well known, see Waldspurger \cite{Wa}:
\begin{equation}\label{e:ind-WF}
{}^{\bar \sk}\WF(\pi)=\ind_{\mathbf M}^{\bfG}({}^{\bar \sk}\WF(\pi_M)),
\end{equation}
where $\ind_{\mathbf M}^{\bfG}$ denotes the Lusztig--Spaltenstein induction of nilpotent orbits. 

We need to use several known facts about the induction of nilpotent orbits.

\begin{lemma}\label{l:prop-nil}
\begin{enumerate}
    \item If $\mathcal O_1\le \mathcal O_2$ are two nilpotent $\mathbf M$-orbits in $\mathfrak m$, then $\ind_{\mathbf M}^{\mathbf G}(\mathcal O_1)\le \ind_{\mathbf M}^{\mathbf G}(\mathcal O_2).$
    \item If $\mathcal O_1^\vee\le \mathcal O_2^\vee$ are two nilpotent $G^\vee$-orbits in $\mathfrak g^\vee$, then $d^\vee(\mathcal O_2^\vee)\le d^\vee(\mathcal O_1^\vee)$.
    \item If $\mathcal O^\vee$ is a nilpotent $M^\vee$-orbits in $\mathfrak m^\vee$, then $d^\vee(G^\vee\cdot \mathcal O^\vee)=\ind_{\mathbf M}^{\mathbf G}(d^\vee(\mathcal O^\vee)).$
\end{enumerate}
\end{lemma}

\begin{proof}
Claim (1) is immediate from the definition, or for example from the formula $\overline{\ind_{\mathbf M}^{\mathbf G}(\mathcal O)}=\mathbf G\cdot (\overline{\mathcal O}+\mathfrak n)$, see \cite[proof of Theorem 7.1.3]{CM}.

Claims (2) and (3) are in \cite[Proposition A2 (a),(c)]{BV}.
\end{proof}

\begin{proposition}
    Conjecture \ref{conj-main-2}(1) implies Conjecture \ref{conj-main}.
\end{proposition}

\begin{proof}
We need to prove that if 
\[{}^{\bar \sk}\WF (\AZ(\pi))\le d^\vee(\bO^\vee_\pi)
\]
holds for all irreducible tempered representations of reductive groups, then it holds for all irreducible representations. Here $\bO^\vee_\pi$ denotes the nilpotent $G^\vee$-orbit of the Langlands parameter of $\pi$. By the parabolic part of the Langlands classification, every irreducible $\pi$ is a quotient of a standard module $i(\mathbf P,\sigma_\nu)=i_{\mathbf P(\sk)}^{\mathbf G(\sk)}(\sigma_\nu)$, $\sigma_\nu=\sigma\otimes\nu$, where $\sigma$ is an irreducible tempered $\mathbf M(\sk)$-representation and $\nu$ is an unramified character of $\mathbf M(\sk)$. 

Since the involution $\AZ$ commutes with parabolic induction \cite[Theorem 1.7]{Au-inv}, i.e., in the Grothendieck group 
\[
\AZ(i_{\mathbf P(\sk)}^{\mathbf G(\sk)}(\sigma_\nu))=i_{\mathbf P(\sk)}^{\mathbf G(\sk)}(\AZ(\sigma)_\nu),
\]
it follows that $\AZ(\pi)$ is a subquotient of $i_{\mathbf P(\sk)}^{\mathbf G(\sk)}(\AZ(\sigma)_\nu)$. By the definition of the wavefront set then
\[
{}^{\bar \sk}\WF (\AZ(\pi))\le {}^{\bar \sk}\WF (i_{\mathbf P(\sk)}^{\mathbf G(\sk)}(\AZ(\sigma)_\nu)={}^{\bar \sk}\WF(i_{\mathbf P(\sk)}^{\mathbf G(\sk)}(\AZ(\sigma))),
\]
where the equality comes from the fact that the local character expansion is valid on neighborhoods of compact elements, and so it does not see twists by unramified characters. By (\ref{e:ind-WF}), it follows that
\[
{}^{\bar \sk}\WF (\AZ(\pi))\le \ind_{\mathbf M}^{\bfG}({}^{\bar \sk}\WF(\AZ(\sigma))).
\]
Now assuming Conjecture \ref{conj-main-2} for tempered representations, we have ${}^{\bar \sk}\WF(\AZ(\sigma))\le d^\vee_{M^\vee}(\bO^\vee_\sigma)$. Then Lemma \ref{l:prop-nil}(1),(3) implies
\[
{}^{\bar \sk}\WF (\AZ(\pi))\le \ind_{\mathbf M}^{\bfG}(d^\vee_{M^\vee}(\bO^\vee_\sigma))=d^\vee(G^\vee\cdot \bO^\vee_\sigma).
\]
Recall that $\pi$ is the Langlands quotient of $i(\mathbf P,\sigma_\nu)$, hence by the desiderata of the Langlands correspondence, $\bO^\vee_\pi=G^\vee\cdot \bO^\vee_\sigma.$
\end{proof}

\section{\bf Depth-zero supercuspidal representations}

Let $\pi$ be a depth-zero supercuspidal representation of $\rG^\omega:=\bfG^\omega(\sk)$. This is realized as follows: there are a vertex $x$ in the Bruhat-Tits building of $\rG^\omega$, thus a maximal parahoric subgroup $\rG^\omega_{x,0}$ and the stabilizer $\rG^\omega_x$ of $x$, a cuspidal representation $\sigma$ of the finite reductive quotient $\C G^\omega_x$ and an irreducible extension $\widetilde\sigma$ of $\sigma$ to $\rG_x^\omega$
such that $\pi\simeq\textrm{c-Ind}_{\rG^\omega_x}^{\rG^\omega}(\widetilde\sigma)$. Write $\pi=\pi(x,\widetilde \sigma)$. Note that $\WF(\pi)$ is independent of the extension $\widetilde\sigma$.

\ 

The finite analogue of the wavefront set for $\sigma$ is the {\it Kawanaka wavefront} set $\WF_q(\sigma)$, \cite{Kaw}. 

\begin{theorem}[{\cite[Theorem 3.0.4]{CMBO-cusp}}]\label{t:supercuspidal-WF} The canonical unramified wavefront set of $\pi(x,\widetilde\sigma)$ is 
\[{}^K\!\underline\WF(\pi(x,\widetilde\sigma))=(\WF(\sigma), x).
\]
\end{theorem}
Lusztig gave an explicit description of $\WF(\mu)$, $\mu$ an irreducible representation of  $\C G^\omega_x$ in terms of his parametrization \cite[Theorem 4.23]{orange}. Let $(\C G^\omega_x)^\vee$ be the dual finite reductive group. Then $\mu$ is parametrized by a triple $(\tau, u_0^\vee,y)$, where $\tau$ is a semisimple element in $(\C G^\omega_x)^\vee$, $u_0^\vee$ is a special unipotent element in the centralizer $C_{(\C G^\omega_x)^\vee}(\tau)$ and $y$ is an element of the unipotent family $C_{u_0^\vee}$ parametrized by $u_0^\vee$. The element $y$ will not be important for us. Let $\widetilde C_{u_0^\vee}=C_{\widetilde u_0^\vee}$ be the family obtained by applying Alvis-Curtis duality to $C_{u_0^\vee}$. Let $d_S^\vee$ be the duality map in this finite field setting, this is defined exactly as above, see \cite[\S13.3]{orange}. 

\begin{theorem}[{\cite[Theorem 11.2]{lusztig-unip-supp}}] With the above notation, $\WF(\mu)=d_S^\vee(\tau \widetilde u_0^\vee)$. If $\mu=\sigma$ is cuspidal, then $\WF(\sigma)=d_S^\vee(\tau u_0^\vee)$.

\end{theorem}
This means that the canonical unramified wavefront set of a supercuspidal representation $\pi(x,\widetilde\sigma)$ is  
\begin{equation}\label{e:CMBO}
{}^K\!\underline\WF(\pi(x,\widetilde\sigma))=(d_S^\vee(\tau u_0^\vee),x).
\end{equation}
Since we are in a depth-zero case, the Langlands parameter $\varphi$ is trivial on the wild inertia group, and the image $\varphi(I_\sk)$ is a cyclic group generated by an element $\lambda\in {}^L\!G$ of finite order prime to $q$. Hence $\fh^\vee={\fg^\vee}^\lambda$ is a pseudo-Levi subalgebra of $\fg^\vee$. Conjecture \ref{conj-cusp} is equivalent to
\begin{equation}
G\cdot d_S^\vee(\tau u_0^\vee)\le d^\vee(G^\vee\cdot \CO_{\lambda,s}^\vee).
\end{equation}

We do not know a uniform way to attach the Langlands parameter to a depth-zero supercuspidal representation, so we will only verify the conjecture in a number of cases that exist in the literature.

\subsection{Regular supercuspidal representations}\label{s:DBR}

For all regular depth-zero supercuspidal representations \cite{KV,DBR}, the conjecture is trivially true since in that case $\fh^\vee$ is just a Cartan subalgebra of $\fg^\vee$, so the only possible Langlands nilpotent element is $0$. Hence the conjecture just says that ${}^{\bar \sk}\WF(\pi)\le \CO_p$, where $\CO_p$ is the principal $\bfG(\bar\sk)$-orbit, which is true, but trivial. On the other hand, ${}^{\bar \sk}\WF(\pi)=\CO_p$ if and only if $\pi$ is a generic representation, and this is not always the case. In general, ${}^{\bar \sk}\WF(\pi)$ can be equal to the $\bfG(\bar\sk)$-saturation of the principal nilpotent orbit in a maximal pseudo-Levi, as we explain next. 

Let $I_t=I_\sk/I^+_\sk$ be the tame quotient of the inertia group, where $I^+_\sk$ is the wild inertia subgroup. The depth-zero condition means that the parameter $\lambda$ in the introduction factors through $I_t$. Write the $L$-group as $^LG=G^\vee\rtimes \langle\theta^\vee\rangle$, where $\theta^\vee$ is a finite-order automorphism giving the inner class of $\bfG$. Suppose $G^\vee$ is semisimple. In the terminology of DeBacker and Reeder, a tame regular semisimple elliptic Langlands parameter is given a $W^\vee$-conjugacy class of pairs $(\lambda,w)$, where:
\begin{enumerate}
    \item a continuous homomorphism $\lambda:I_t\to T^\vee$ such that $Z_{G^\vee}(\lambda)=T^\vee$, and
    \item a Weyl group element $w\in W^\vee=N_{G^\vee}(T^\vee)/T^\vee$ such that $w^\vee\circ \theta^\vee\circ \lambda^q=\lambda$ and ${T^\vee}^{w^\vee\theta^\vee}$ is finite. 
\end{enumerate}
The element $w^\vee\theta^\vee$ is the image of the inverse Frobenius under the Langlands homomorphism $\varphi$. Then a complete (enhanced) Langlands parameter is given by a triple $(\lambda,w,\rho)$, where $\rho$ is an irreducible representation of the finite group $(T^\vee)^{w^\vee\theta^\vee}=Z_{G^\vee}(\varphi).$ If $X$ is the weight lattice of $\bfG$, and $W$ is the Weyl group of $\bfG$, $\theta$ the Galois automorphism giving $\theta^\vee$, there is a natural isomorphism
\begin{equation}
X/(1-w\theta)X\longleftrightarrow \Irr~(T^\vee)^{w^\vee\theta^\vee}, \quad x\mapsto \rho_x.
\end{equation}
By abuse of notation, denote $x$ also the preimage in $X$ of an element in $X/(1-w\theta)X$. Let $\pi(\lambda, w,0,x)$ denote the supercuspidal representation attached to $(\lambda,w,x)$ in {\it loc. cit}. The $0$ here stands for the zero nilpotent orbit, i.e., the image of the $SL(2)$ under the Langlands parameter must be trivial (since $(\fg^\vee)^{\lambda(I_\sk)}=\mathfrak t^\vee$). 

\begin{proposition}
    The geometric wavefront set ${}^{\bar\sk}\WF(\pi(\lambda,w,0,x))$ is the $\bfG(\bar\sk)$-saturation of the principal nilpotent orbit of the maximal pseudo-Levi subgroup $Z_{\bfG(\bar\sk)}(x).$
\end{proposition}

\begin{proof}
    The claim follows at once, after we recall the construction of $\pi(\lambda,w,0,x)$. The element $x$ defines a cocycle $u_x\in H^1(\sk, \bfG)$ which gives the inner twist  $\omega_x$ of $\bfG$ and $F_x$ a twisted Frobenius map and a maximal parahoric subgroup $\bfG_x$. Then $\bfG_x^{F_x}$ is a maximal parahoric subgroup of $\bfG^{\omega_x}(\sk)$. The pair $(w,x)$ defines an $F_x$-minisotropic torus $T_x$ and $\lambda$ defines a depth-zero character $\chi_x$ of $T_x$. Then Deligne-Lusztig induction defines an irreducible cuspidal representation of the reductive quotient $\overline {\bfG_x^{F_x}}$
    \[
    \kappa_x=\pm R_{T_x}^{\overline {\bfG_x^{F_x}}}(\chi_x).
    \]
    The important thing for us is that this cuspidal representation is generic, in the sense that the Kawanaka wavefront set is principal. The supercuspidal representation is compactly induced
    \[
    \pi(\lambda,w,0,x)=\operatorname{ind}_{\mathbf Z^{F_x}\bfG_x^{F_x}}^{\bfG^{\omega_x}(\sk)}(\widetilde\kappa_x).
    \]
   The wavefront set is given by Theorem \ref{t:supercuspidal-WF}. 
\end{proof}
Conjecture \ref{conj-cusp} clearly holds. Moreover, so does Conjecture \ref{conj-main-2}(2) since the $L$-packet of supercuspidals contains a generic supercuspidal representation.

\section{\bf Depth-zero supercuspidal representations of classical groups}\label{s:classical} 

In \cite{LS}, Lust and Stevens constructed $L$-packets of depth-zero supercuspidal representations of classical groups. We will verify Conjecture 1.6 for these $L$-packets. 
Let $\C G$ be a connected reductive group of classical type 
over $\bF_q$, $q$ odd. Let $F$ be the Frobenius map. Let $\C G^*$ be the dual group over $\bF_q$ with dual Frobenius also denoted $F$. Let $\Irr\, \C G^F$ be the set of equivalence classes of irreducible complex $\C G^F$-representations. According to Lusztig's classification, we can partition
\[
\Irr\, \C G^F=\sqcup_s \Irr_s \C G^F,
\]
where $s$ runs over the conjugacy classes of semisimple elements in ${\C G^*}^F$. The piece $\Irr_1 \C G^F$ is the set of irreducible unipotent representations. Since $Z(\C G)$ is connected, the  rational semisimple conjugacy classes in $\C G^*$ are in one-to-one correspondence with the geometric $F$-stable semisimple conjugacy classes in $\C G^*$ \cite[Proposition 3.7.3]{Car}.

The above decomposition is realized as follows. By \cite[Theorem 4.4.6]{Car}, there exists a natural one-to-one correspondence between geometric conjugacy classes of pairs $(T,\theta)$, where $\C T$ is a maximal $F$-stable torus in $\C G$ and $\theta$ is a character of $T$, and $F$-stable semisimple conjugacy classes in $\C G^*$. Let $s\in \C T^*\subset \C G^*$, then $\Irr_s \C G^F$ consists of all the constituents of the induced representations $R_{\C T}^{\C G}(s)$ where we think of $s$ as a character of $T$.

Lusztig's Jordan decomposition says that there is a one-to-one correspondence
\[
\Psi_s^{\C G}:\Irr_s \C G^F\longleftrightarrow \Irr_1(\C G_s^F),
\]
where $\C G_s$ is the endoscopic group of $\C G$ with dual $\C G^*_s:=Z_{\C G^*}(s)$. To have cuspidal representations, we need $Z(\C G^*)^0$ and $Z(\C G_s^*)^0$ to have the same $F$-rank. In this case, $\Psi_s^{\C G}$ gives a correspondence of cuspidal representations. In other words, the set of cuspidal $\C G^F$-representations are in one-to-one correspondence with $\C G^*$-pairs $(s,\tau)$ where $s\in {\C G^*}^F$ is semisimple of equal rank and $\tau$ is a cuspidal unipotent representation of $\C G_s^F$.

For classical groups, the classification of cuspidal unipotent representations shows that there is at most one cuspidal unipotent representation. On the other hand, the conjugacy classes of $s\in {\C G^*}^F$ are given by characteristic polynomials $P_s\in \bF_q[x]$,
\[
P_s(x)=\prod_P P(x)^{a_P},
\]
where the product runs over all irreducible self-dual monic polynomials over $\bF_q$ and nonnegative integers $a_P$ such that $\sum_P a_P \deg(P)=\dim V$, where $V$ is the defining $\mathbb F_q$-representation of $\C G^*$. In order for the semisimple class to parametrize a cuspidal representation, one needs:
\begin{enumerate}
\item if $P(x)\neq x\pm 1$, then it has even degree $n_p=2k_P$, $k_P\ge 1$, and multiplicity $a_P=\frac 12(m_P^2+m_P)$ for some integer $m_P\ge 0.$
\item if we set $a_\pm=a_{x\mp 1}$, there are integers $m_\pm\ge 0$ with conditions, depending on the type of the group, see \cite[\S7.2]{LS}, as follows:
\begin{enumerate}
\item if $\C G=SO(2n+1)$, then $a_+=2(m_+^2+m_+)$ and $a_-=2(m_-^2+m_-)$;
\item if $\C G=GSO^\pm(2n)$, then $a_+=2m_+^2$ and $a_-=2m_-^2$;
\item if $\C G=Sp(2n)$, then $a_+=2(m_+^2+m_+)+1$ and $a_-=2m_-^2$.
\end{enumerate}
\end{enumerate}
\begin{example} A cuspidal representation of $Sp(2n,\bF_q)$ is uniquely determined by a collection of nonnegative integers $m_\pm$, $m_P,k_P$ such that
\begin{equation}
    \sum_{\deg(P)>1} k_P\frac {m_P^2+m_P}2+(m_+^2+m_+)+ m_-^2=n.
\end{equation}
The $\deg(P)>1$ terms, if they appear, correspond to endoscopic unitary groups, whereas the $m_+$ term corresponds to an endoscopic $Sp$ subgroup and $m_-$ to an endoscopic even orthogonal group.
\end{example}

\subsection{Nilpotent Langlands parameters}\label{s:nil-L} Let $G$ be a classical group over a $p$-adic field $\sk$. By this we mean, a symplectic, special orthogonal, or a unitary group. Every depth-zero supercuspidal representation is compactly-induced 
\begin{equation}
    \pi=\text{ind}_K^G(\tau_1\boxtimes\tau_2),
\end{equation}
where $K$ is the normalizer of a maximal parahoric subgroup with reductive quotient $\bar K=\CG^\circ_{n_1,n_2}$ over $\bF_q$, and $\tau_1\boxtimes \tau_2$ is a cuspidal representation of $\CG^\circ_{n_1,n_2}(\bF_q)$.  The explicit cases are as follows \cite[\S2]{LS}, for more details about the notation, see {\it loc. cit.}:
\begin{enumerate}
    \item $G$ symplectic, $\CG^\circ_{n_1,n_2}(\bF_q)=Sp(2n_1,\bF_q)\times Sp(2n_2,\bF_q)$;
    \item $G$ special orthogonal, $\CG^\circ_{n_1,n_2}(\bF_q)=SO(n_1+n_1^{an},n_1,\bF_q)\times SO(n_2+n_2^{an},n_2,\bF_q)$, where $SO(m+m^{an},m,\bF_q)$ denotes the special orthogonal group with form of Witt index $m$ and anisotropic part of dimension $m^{an}\le 2$;
    \item $G$ unramified unitary, $\CG^\circ_{n_1,n_2}=U(2n_1+n_1^{an})\times U(2n_2+n_2^{an})$, a product of two unitary groups over $\bF_q$;
    \item $G$ ramified unitary, then $\CG^\circ_{n_1,n_2}$ is
    \[
    SO(n_1+n_1^{an},n_1,\bF_q)\times Sp(2n_2,\bF_q)\text{ or }Sp(2n_1,\bF_q)\times SO(n_2+n_2^{an},n_2,\bF_q).
    \]
\end{enumerate}
Let $(m_\pm^{(i)},m_P^{(i)},k_P^{(i)})$, $i=1,2$ be the collections of integers that characterizes $\tau_i$. The nilpotent Langlands parameter of $\pi$ is a nilpotent orbit $\bO^\vee_\pi$ in $\fg^\vee$, which is either $gl(n,\bC)$, $sp(2n,\bC)$ or $so(m,\bC)$. Via the Jordan normal form, $\bO^\vee_\pi$ is parametrized by a partition $\lambda$ of $n$, $2n$, or $m$, respectively. If $\fg^\vee=so(m,\bC)$, then the partition is orthogonal: each even part appears with even multiplicity. If 
$\fg^\vee=sp(2n,\bC)$, then the partition is symplectic: each odd part appears with even multiplicity. 

Lust and Stevens \cite[\S1 and \S8]{LS} attach to $\pi$ the nilpotent orbit $\bO^\vee_\pi$ given by the partition obtained by the following procedure. For every integer $\alpha$ denote by $[\alpha]$ the string $\{\alpha,\alpha-2,\alpha-4,\dots,1\}$ if $\alpha>0$ is odd, $\{\alpha,\alpha-2,\alpha-4,\dots,2\}$ if $\alpha>0$ is even, and $[\alpha]=\emptyset$ if $\alpha\le 0$. Start with an empty string and repeat the following procedure for each $P$:
\begin{enumerate}
\item if $n_P\ge 2$, append $n_P$ copies of $[m_P^{(1)}+m_P^{(2)}]$ and $n_P$ copies of $[|m_P^{(1)}-m_P^{(2)}|-1]$.
\end{enumerate}
For {\bf unramified unitary groups}, this completes the nilpotent Langlands parameter. For the other classical groups, we need to append more strings defined in terms of $m_\pm^{(i)}$ and the type of the $p$-adic group.  

\medskip

{\bf Symplectic groups.} 

\begin{enumerate}
\item[(2s)] append the strings $[2(m_+^{(1)}+m_+^{(2)})+1]$ and $[2|m_+^{(1)}-m_+^{(2)}|-1]$;
\item[(3s)] append the strings $[2(m_-^{(1)}+m_-^{(2)})-1]$ and $[2|m_-^{(1)}-m_-^{(2)}|-1]$.
\end{enumerate}

\medskip

{\bf Ramified unitary groups.}  Without loss of generality, assume $\CG^\circ_{n_1,n_2}(\bF_q)=Sp(2n_1,\bF_q)\times SO(n_2+n_2^{an},n_2,\bF_q)$. There are two cases depending on the parity of $n_2^{an}$:

\begin{enumerate}
\item[(2r)] append the strings $[2(m_+^{(1)}+m_+^{(2)})+1]$ and $[2|m_+^{(1)}-m_+^{(2)}|-1]$ if $n_2^{an}=1$, or the strings $[2(m_+^{(1)}+m_+^{(2)})]$ and $[2|m_+^{(1)}-m_+^{(2)}+1|-1]$ if $n_2^{an}$ is even;
\item[(3r)] append the strings $[2(m_-^{(1)}+m_-^{(2)})]$ and $[2|m_-^{(1)}-m_-^{(2)}-1|-1]$ if $n_2^{an}=1$, or the strings $[2(m_-^{(1)}+m_-^{(2)})-1]$ and $[2|m_-^{(1)}-m_-^{(2)}|-1]$ if $n_2^{an}$ is even.
\end{enumerate}

\medskip

{\bf Special orthogonal groups.} There are three cases depending on the parities of $n_1^{an}$ and $n_2^{an}$.

\begin{enumerate}
    \item[(2o)] 
    \begin{itemize}\item If $n_1^{an}=1$ and $n_2^{an}$ is even, append the strings $[2(m_+^{(1)}+m_+^{(2)})]$ and $[|2(m_+^{(1)}-m_+^{(2)})+1|-1]$. 

    \item If $n_1^{an}$ and $n_2^{an}$ are both even, append the strings $[2(m_+^{(1)}+m_+^{(2)})-1]$ and $[|2(m_+^{(1)}-m_+^{(2)})|-1]$.

    \item If $n_1^{an}=n_2^{an}=1$, append the strings $[2(m_+^{(1)}+m_+^{(2)})+1]$ and $[|2(m_+^{(1)}-m_+^{(2)})|-1]$.
\end{itemize}
    \item[(3o)] Exactly the same, but with $m_-$ in place of $m_+$.
\end{enumerate}

\subsection{Partitions} We record several elementary facts about partitions that will be used in the proofs in the next subsections.

\begin{definition}\label{d:sum-parts}
    If $\alpha,\beta$ are two partitions of $m,n$, respectively, the sum $\alpha+\beta$ is a partition of $m+n$ given as follows. Suppose $\alpha=(\alpha_1,\alpha_2,\dots,\alpha_\ell)$ and $\beta=(\beta_1,\beta_2,\dots,\beta_\ell)$ with $\alpha_1\ge\alpha_2\ge\dots\ge\alpha_\ell\ge 0$, $\beta_1\ge\beta_2\ge\dots\ge\beta_\ell\ge 0$. (By allowing entries of $0$, we may assume there is the same number of parts in $\alpha$ and $\beta$.) Then $\alpha+\beta=(\alpha_1+\beta_1,\alpha_2+\beta_2,\dots,\alpha_\ell+\beta_\ell).$
\end{definition}

    The above operation on partitions plays a key role in the explicit computation of $j$-induction for representations of the Weyl group of classical types.

\smallskip

If $\alpha$ and $\beta$ are two partitions, denote by $\alpha\sqcup \beta$ the partition obtained by concatenation (and reordering). If $\alpha$ is a partition, let $\alpha^t$ denote the transpose partition. Let $\alpha\le \beta$ denote the dominance order of partitions.

\begin{lemma}\label{l:partition-facts}

\begin{enumerate}
\item If $\alpha\le \beta$ then $\alpha^t\ge\beta^t$.
    \item If $\alpha_1\le\beta_1$ and $\alpha_2\le \beta_2$ then $\alpha_1\sqcup\alpha_2\le\beta_1\sqcup\beta_2$.
    \item For any two partitions $\alpha,\beta$, 
\begin{equation}\label{compatibility}
(\alpha\sqcup\beta)^t=\alpha^t+\beta^t
\end{equation}
\item For any four partitions $\alpha_1,\alpha_2,\beta_1,\beta_2$,
\begin{equation}\label{compatibility-2}
(\alpha_1\sqcup\beta_1)+(\alpha_2\sqcup\beta_2)\ge (\alpha_1+\alpha_2)\sqcup (\beta_1+\beta_2).
\end{equation}
\item For any three partitions $\alpha,\beta,\gamma$, if $\alpha \ge  \beta$, then $\alpha+\gamma\ge \beta +\gamma$.
   
\end{enumerate}

\end{lemma}

\begin{proof}
    Claims (1), (2), (3) are elementary and well known, see \cite[\S5.1]{CMBO23} for example. Claim (5) follows from (2) and (3). We prove (4).

    Let us first prove the case of three partitions:
    \[
\beta_1+(\alpha\sqcup\beta_2)\ge \alpha\sqcup (\beta_1+\beta_2).
\]
    by induction on the length of $\alpha$. If $\alpha=\emptyset$, this is clearly true. Suppose $\alpha=\{s\}$ is a singleton. Write $\beta_1=(a_1\ge a_2\ge\dots\ge a_k)$ and $\beta_2=(b_1\ge b_2\ge\dots\ge b_k)$ and suppose $i$ and $t$ are the smallest indices such that $b_i>s\ge b_{i+1}$ and $a_t+b_t>s\ge a_{t+1}+b_{t+1}$. (We think $a_0=b_0=\infty$ for the consistency of the notation.)
    Then
    \[
    \beta_1+(\beta_2\sqcup \{s\}):\  a_1+b_1\ge a_2+b_2\ge\dots \ge a_i+b_i\ge a_{i+1}+s\ge a_{i+2}+b_{i+1}\ge\dots\ge a_k+b_{k-1}\ge b_k,
    \]
    while 
    \[
    (\beta_1+\beta_2)\sqcup\{s\}:\  a_1+b_1\ge a_2+b_2\ge\dots \ge a_t+b_t> s\ge a_{t+1}+b_{t+1}\ge\dots\ge a_{k}+b_k.
    \]
    and it is immediate that $\beta_1+(\beta_2\sqcup \{s\})\ge (\beta_1+\beta_2)\sqcup\{s\}$.

    Now suppose the inequality holds for $\alpha$ and consider $\wti\alpha=\alpha\sqcup\{s\}$. Denote $\wti\beta_2=\beta_2\sqcup\{s\}$. Then
    \[
    \beta_1+(\wti\alpha\sqcup\beta_2)=\beta_1+(\alpha\sqcup \wti\beta_2)\ge \alpha\sqcup(\beta_1+\wti\beta_2)\ge \alpha\sqcup\{s\}\sqcup (\beta_1+\beta_2)=\wti\alpha\sqcup (\beta_1+\beta_2),
    \]
    where the first inequality follows by the inductive step and the second by the case of a singleton. This proves the case of three partitions.

    The inequality for four partitions (\ref{compatibility-2}) follows by induction on the length of $\alpha_1$. If $\alpha_1=\emptyset$, this is the case of three partitions. Suppose $\alpha_1=(a_1\ge a_2\ge\dots\ge a_k>0)$ is a partition and let $\bar\alpha_1=(a_1\ge a_2\ge\dots\ge a_{k-1})$ so that $\alpha_1=\bar\alpha_1\sqcup\{a_k\}$. By induction 
    \[
    (\bar\alpha_1\sqcup\beta_1)+(\alpha_2\sqcup\beta_2)\ge (\bar\alpha_1+\alpha_2)\sqcup (\beta_1+\beta_2).
    \]
    Let us look at the contribution of $a_k$ on both sides. On the right hand side $a_k$ is added to the $k$-th part $a_k'$ of $\alpha_2$ (written in decreasing order). In the left hand side, if we write {$\alpha_2\sqcup\beta_2=(c_1\ge c_2\ge\dots)$}, then $a_k$ is added to $c_j$  $j\ge k$, such that $c_i\ge a_k'$ for all $i\le j$. This means that the addition of $a_k$ in the two sides preserves the dominance order from before, and the claim follows inductively.
\end{proof}

\subsection{Symplectic groups}
We verify Conjecture \ref{conj-cusp} for the case of depth-zero supercuspidal representations of $G=Sp(2n,\sk)$, by computing the wavefront set and compare it with the nilpotent Langlands parameter defined above. 

At one extreme suppose that only endoscopic unitary factors appear, that is only polynomials $P$ with $\deg P>1$. Hence $\bO^\vee_\pi$ is obtained by appending $2k_P$ copies of $[m_P^{(1)}+m_P^{(2)}]$ and $2k_P$ copies of $[|m_P^{(1)}-m_P^{(2)}|-1]$, for each $k_P>0$. On each $Sp(2n_i)$, $i=1,2$, factor, we have 
\[
\sum_{\deg P>1} k_p a_p^{(i)}=n_i,\quad a_P^{(i)}=\frac{{m_P^{(i)2}}+m_P^{(i)}}2.
\]
To compute the wavefront set, we need to compute the $j$-induction on each factor. To simplify the notation, let us drop temporarily the indices $i$. Let $W_n$ be the Weyl group of $G^\vee$. Then, the calculation amounts to determining the irreducible $W_n$-representation given by
\[
j_{\prod_P S_{a_P}^{k_P}}^{W_n}\left(\boxtimes_P (m_P,\dots,2,1)^{\boxtimes k_P}\right).
\]
Using the formulae for $j$-induction \cite[\S3.1,\S4.4]{lusztig-j-ind}, this equals
\begin{align}
    j_{\prod_P S_{k_Pa_P}}^{W_n}\left(\boxtimes_P(m_Pk_P,\dots,2k_P,k_P)\right)=j_ {\prod_P W_{k_Pa_P}}^{W_n}(\sigma_P),
\end{align}
where
\begin{equation*}
   \sigma_P=\begin{cases} (m_P k_P,\dots,3k_P,k_P)\times ((m_P-1)k_P,\dots,4k_P,2k_P),&m_P\text{ odd}\\
   (m_P k_P,\dots,4k_P,2k_P)\times ((m_P-1)k_P,\dots,3k_P,k_P),&m_P\text{ even}.
   \end{cases}
\end{equation*}

A direct calculation with the algorithm for the Springer correspondence gives that the type $C$ nilpotent orbit for $\sigma_P$ is given by the partition
\begin{equation}\label{e:Springer1}
\sigma_P\xrightarrow{\text{Springer}} (2m_Pk_P,\dots,4k_P,2k_P).
\end{equation}
This is true in both cases, $m_P$ even or odd.

\begin{lemma}\label{l:j-ind}
    The irreducible $W_n$-representation which is $j$-induced \[j_{\prod_P S_{a_P}^{k_P}}^{W_n}\left(\boxtimes_P (m_P,\dots,2,1)^{\boxtimes k_P}\right)\]
    is attached in the Springer correspondence to the symplectic nilpotent orbit
    \[
    \sum_P (2m_Pk_P,\dots,4k_P,2k_P),
    \]
    where the sum is in the sense of Definition \ref{d:sum-parts}.
\end{lemma}

\begin{proof}
    By the calculation leading up to (\ref{e:Springer1}), we need to show that the nilpotent orbit for $j_ {\prod_P W_{k_Pa_P}}^{W_n}(\boxtimes_P\sigma_P)$ is as claimed in the statement of the lemma. Since $j$-induction works in stages, this comes down to the case of two polynomials $P_1,P_2$, that is to the easy verification that $j_{W_1\times W_2}^{W}(\sigma_{1}\boxtimes \sigma_{2})$ is the Springer representation for the nilpotent $(2m_1k_1,\dots,4k_1,2k_1)+(2m_2k_2,\dots,4k_2,2k_2)$.
\end{proof}

\begin{proposition}\label{p:unitary}Let $\pi$ be a depth-zero supercuspidal $Sp(2n,\mathsf k)$-representation as above.    Suppose there are only endoscopic unitary factors, i.e., only polynomials $P$ with $\deg P>1$ appear. Then 
\begin{enumerate}
\item $\bO^\vee_\pi=[1]\sqcup \bigsqcup_P [m_P^{(1)}+m_P^{(2)}]^{2k_P}\sqcup [|m_P^{(1)}-m_P^{(2)}|-1]^{2k_P}$.
\item \[{}^{\bar{\mathsf k}}\WF(\pi)=\bigsqcup_{i=1,2}\left(\sum_P(2m_P^{(i)} k_p,\dots,4k_P,2k_P)\right)
\]
\item Conjecture \ref{conj-cusp} holds in this case.
\end{enumerate}
    
\end{proposition}

\begin{proof}
Claim (1) is just a restatement of the Lust--Stevens parametrization. Notice the presence of the entry $[1]$ because of the step (2s) in the procedure. Claim (2) is directly implied by Lemma \ref{l:j-ind}. It remains to verify (3).

Suppose first that $n_1=n$, $n_2=0$, this is the hyperspecial case. Write $m_P$ for $m_P^{(1)}$. Then $\mathbb O^\vee_\pi$ is simply
\[
\bO^\vee_\pi=[1]\sqcup\bigsqcup_P (\underbrace{m_P,\dots,m_P}_{2k_P},\underbrace{m_P-1,\dots,m_P-1}_{2k_P},\dots,\underbrace{1,\dots,1}_{2k_P}),
\]
while
\[{}^{\bar{\mathsf k}}\WF(\pi)=\sum_P(2m_P k_p,\dots,4k_P,2k_P)\]
The combinatorial formula for computing $d^\vee$ on partitions $\underline p$ of $2n+1$ of odd orthogonal type is:
\begin{equation}\label{e:duality-C}
d^\vee(\underline p)=((\underline p^-)_C)^t=((\underline p^t)^-)_C,
\end{equation}
where:
\begin{itemize}
    \item $\underline p\to\underline p^-$ decreases the smallest entry in $\underline p$ by $1$, thus making $\underline p^-$ a partition of $2n$;
    \item $\underline p^-\to (\underline p^-)_C$ is the $C$-collapse, resulting in a symplectic partition, see \cite[\S6.3]{CM};
    \item $^t$ is the transpose partition.
\end{itemize}
In our case, notice that if we denote the partition of $\mathbb O^\vee_\pi$ by $\underline p$, then $\underline p^-$ simply removes the entry $1$, and the $C$-collapse doesn't do anything since $\underline p^-$ is already a symplectic partition. 

Moreover \begin{align*}
(\underbrace{m_P,\dots,m_P}_{2k_P},\underbrace{m_P-1,\dots,m_P-1}_{2k_P},\dots,\underbrace{1,\dots,1}_{2k_P})^t
=(2m_P k_p,\dots,4k_P,2k_P)\end{align*}
which proves that $d^\vee(\bO_\pi^\vee)={}^{\bar{\mathsf k}}\WF(\pi)$ if there is a single $P$. If there is more than one polynomial $P$, the claim follows as well by the compatibility (\ref{compatibility})
for two (or more) partitions $\alpha,\beta$.

\medskip

Now suppose both $n_1,n_2>0$. The key observation is that for all $k$ and $m^{(1)}\ge m^{(2)}$,
\begin{align*}
d^\vee([1]\sqcup [m^{(1)}+m^{(2)}]^{2k}\sqcup [|m^{(1)}-m^{(2)}|-1]^{2k})&=\left((m^{(1)}+m^{(2)},m^{(1)}+m^{(2)}-2,\dots,m^{(1)}-m^{(2)}+2\right.,\\&\left. m^{(1)}-m^{(2)},m^{(1)}-m^{(2)}-1,\dots,2,1)^{2k}\right)^t\\
&=(2m^{(1)}k,\dots,4k,2k)\sqcup (2m^{(2)}k,\dots,4k,2k),
\end{align*}
which proves (3) (with equality in Conjecture \ref{conj-cusp}) in the case of a single polynomial $P$. 

Finally, in the general case with more polynomials, we use again the compatibility (\ref{compatibility}) and see that
\begin{align*}
d^\vee(\bO^\vee_\pi)&=\left(\bigsqcup_P [m_P^{(1)}+m_P^{(2)}]^{2k_P}\sqcup [|m_P^{(1)}-m_P^{(2)}|-1]^{2k_P}\right)^t\\
&=\sum_P \left((2m_P^{(1)} k_P,\dots,4k_P,2k_P)\sqcup (2m_P^{(2)} k_P,\dots,4k_P,2k_P)\right)
\end{align*}
Using Lemma \ref{l:partition-facts}(5) and applying the inequality (\ref{compatibility-2}) repeatedly, one can see that $d^\vee(\bO^\vee_\pi)$ is greater in the partition order than ${}^{\bar{\mathsf k}}\WF(\pi)=\bigsqcup_{i=1,2}\left(\sum_P(2m_P^{(i)} k_p,\dots,4k_P,2k_P)\right)$. Notice that the difference between the two formulae is simply that $\Sigma$ and $\sqcup_{i=1,2}$ are interchanged.

To see that the inequality ${}^{\bar{\mathsf k}}\WF(\pi)\le d^\vee(\bO^\vee_\pi)$ can be strict, consider the case when there are just two distinct polynomials $P_1, P_2$, where $P_i$ occurs on the $Sp(2n_i)$ factor. Then 
\[
{}^{\bar{\mathsf k}}\WF(\pi)=(2m_{P_1}^{(1)} k_{P_1},\dots,4k_{P_1},2k_{P_1})\sqcup (2m_{P_2}^{(2)} k_{P_2},\dots,4k_{P_2},2k_{P_2})
\]
while
\[
d^\vee(\bO^\vee_\pi)=(2m_{P_1}^{(1)} k_{P_1},\dots,4k_{P_1},2k_{P_1})+(2m_{P_2}^{(2)} k_{P_2},\dots,4k_{P_2},2k_{P_2}).
\]

\end{proof}

The other extreme is that only endoscopic factors of type $B$ and $D$ occur in the construction of the cuspidal representation on the finite reductive group of the parahoric $Sp(2n_1)\times Sp(2n_2)$.

\begin{proposition}\label{p:no-unitary}
Let $\pi$ be a depth-zero supercuspidal $Sp(2n,\mathsf k)$-representation as above.    Suppose there are no endoscopic unitary factors, i.e., no polynomials $P$ with $\deg P>1$ appear. Then 
\begin{enumerate}
\item $\bO^\vee_\pi=[2(m_+^{(1)}+m_+^{(2)})+1]\sqcup [2|m_+^{(1)}-m_+^{(2)}|-1]\sqcup [2(m_-^{(1)}+m_-^{(2)})-1]\sqcup [2|m_-^{(1)}-m_-^{(2)}|-1].$
\item \[{}^{\bar{\mathsf k}}\WF(\pi)=\bigsqcup_{i=1,2}[2(m_+^{(i)}+m_-^{(i)})]\sqcup \begin{cases}[2(m_+^{(i)}-m_-^{(i)})],&m_+^{(i)}\ge m_-^{(i)},\\
[2(m_-^{(i)}-m_+^{(i)}-1)],&m_+^{(i)}< m_-^{(i)}
\end{cases}.
\]
\item Conjecture \ref{conj-cusp} holds in this case.
\end{enumerate}
\end{proposition}

\begin{proof}

Claim (1) is just a restatement of the above procedure by Lust--Stevens, listed here for convenience. Claim (2) follows from formula (\ref{e:CMBO}) by a direct calculation using $j$-induction. More precisely, ${}^{\bar{\mathsf k}}\WF(\pi)$ is the saturation (concatenation in terms of the partitions) of the two Kawanaka wavefront set nilpotent orbits on the two factors $Sp(2n_i)$ of the maximal parahoric. 

Fix one factor and drop the index $(i)$ from the notation for simplicity. The Kawanaka wavefront set is the symplectic nilpotent orbit for which the Springer representation (attached to the trivial local system) is given the $j$-induced
\[
j_{W(B)\times {W(D)}}^{W_n}\left( (m_+,\dots,2,1)\times (m_+,\dots,2,1))\boxtimes (m_-,\dots,2,1)\times (m_--1,\dots,2,1))\right),
\]
where the notation indicates that the first factor is of type $B$ and the second of type $D$. We use the customary bipartition notation $\alpha\times \beta$ for an irreducible representation of the Weyl groups of types $B$ or $D$. 

Using the formulae for $j$-induction from \cite[\S5]{lusztig-j-ind}, we see this $j$-induced representation is the bipartition
\[
\left((m_+,\dots,2,1)+(m_-,\dots,2,1)\right)\times \left((m_+,\dots,2,1)+(m_--1,\dots,2,1)\right),
\]
where $+$ has the meaning from Definition \ref{d:sum-parts}. Applying the combinatorial algorithm for the Springer representation for {\it type $C$},
we find that the corresponding {\it symplectic} nilpotent orbit is
\[
[2(m_++m_-)]\sqcup \begin{cases}[2(m_+-m_-)],&m_+\ge m_-,\\
[2(m_--m_+-1)],&m_+< m_-
\end{cases},
\]
and Claim (2) is proven.

Claim (3) follows by computing Spaltenstein duality $d^\vee(\bO^\vee_\pi)$ using formulas (\ref{e:duality-C}) and (\ref{compatibility}). We begin with the hyperspecial case: $n_2=0$, and $m_+^{(2)}=m_-^{(2)}=0$, and explain how we get an equality $d^\vee(\bO^\vee_\pi)={}^{\bar{\mathsf k}}\WF(\pi)$.

For simplicity, denote $a=m_+^{(1)}$, $b=m_-^{(1)}$. Then the partition for $\mathbb O_\pi^\vee$ is
\[\underline p=[2a+1]\sqcup [2a-1]\sqcup [2b-1]\sqcup [2b-1].
\]
Firstly, notice that
\[
(2a+1,2a-1,2a-3,\dots,3,1)^t=(a+1,a,a,a-1,a-1,\dots,2,2,1,1),
\]
an identity that we will use repeatedly in the proof.
Next, using (\ref{compatibility}), we get
\begin{equation}\label{e:case-1}
\begin{aligned}
\underline p^t&=(a+1,a,a,a-1,a-1,\dots,1,1)+(a,a-1,a-1,\dots,1,1)+(2b,2(b-1),2(b-1),\dots,4,4,2,2)\\
&=(2a+1,2a-1,2a-1,\dots,3,3,1,1)+(2b,2(b-1),2(b-1),\dots,4,4,2,2)\\
&=\begin{cases}
    (2(a+b)+1,2(a+b)-3,2(a+b)-3,\dots,2(a-b)+5,2(a-b)+5,&\\\ \ \ \ \ \ \ \ \ \  2(a-b)+1,2(a-b)+1,\dots,3,3,1,1),& \ \ \ a\ge b,\\
    (2(a+b)+1,2(a+b)-3,2(a+b)-3,\dots,2(b-a)+1,2(b-a)+1,&\\\ \ \ \ \ \ \ \ 2(b-a)-2,2(b-a)-2,\dots,4,4,2,2),& \ \ \ a<b.
\end{cases}
\end{aligned}
\end{equation}
Now it is easy to see that removing a $1$ and applying the $C$-collapse, we get
\[
d^\vee(\mathbb O_\pi^\vee)=((\underline p^t)^-)_C=\begin{cases}
[2(a+b)]\sqcup [2(a-b)], &a\ge b,\\
[2(a+b)]\sqcup [2(b-a)-2], &a<b,
\end{cases}
\]
and this is exactly the partition for ${}^{\bar{\mathsf k}}\WF(\pi)$.

\medskip

It remains to prove the inequality $d^\vee(\mathbb O_\pi^\vee)\ge {}^{\bar{\mathsf k}}\WF(\pi)$ for arbitrary $m_\pm^{(i)}$. 

Denote $\underline p=[2(m_+^{(1)}+m_+^{(2)})+1]\sqcup [2|m_+^{(1)}-m_+^{(2)}|-1]\sqcup [2(m_-^{(1)}+m_-^{(2)})-1]\sqcup [2|m_-^{(1)}-m_-^{(2)}|-1].$ Using (\ref{compatibility}), we have
\begin{align*}
\underline p^t&=(m_+^{(1)}+m_+^{(2)}+1,m_+^{(1)}+m_+^{(2)},m_+^{(1)}+m_+^{(2)},\dots,1,1)+(|m_+^{(1)}-m_+^{(2)}|,|m_+^{(1)}-m_+^{(2)}|-1,|m_+^{(1)}-m_+^{(2)}|-1,\dots,1,1)+
\\&(m_-^{(1)}+m_-^{(2)},m_-^{(1)}+m_-^{(2)}-1,m_-^{(1)}+m_-^{(2)}-1,\dots,1,1)+(|m_-^{(1)}-m_-^{(2)}|,|m_-^{(1)}-m_-^{(2)}|-1,|m_-^{(1)}-m_-^{(2)}|-1,\dots,1,1).
\end{align*}
This is a sum of four partitions. Set $\gamma_+$ to be the sum of the first two, and $\gamma_-$ the sum of the last two. Then $\underline p^t=\gamma_++\gamma_-$, 
\begin{align*}
\gamma_+&=(2M_++1,2M_+-1,2M_+-1,2M_+-3,2M_+-3,\dots,1,1)\sqcup(2m_+,2m_+,2m_+-2,2m_+-2,\dots,2,2),\\
\gamma_-&=(2M_-,2M_--2,2M_--2,2M_--4,2M_--4,\dots,2,2)\sqcup(2m_--1,2m_--1,2m_--3,2m_--3,\dots,1,1),
\end{align*}
where $M_\pm=\max\{m_\pm^{(1)},m_\pm^{(2)}\}$, and $m_\pm=\min\{m_\pm^{(1)},m_\pm^{(2)}\}.$ Denote the two partitions in $\gamma_+$ by $\gamma_+^M$ and $\gamma_+^m$, respectively, so $\gamma_+=\gamma_+^M\sqcup \gamma_+^m$, and similarly for $\gamma_-$.

We need to compare $((\gamma_++\gamma_-)^{-})_C$ with the partitions describing the geometric wavefront set. We will use inequality (\ref{compatibility-2}).
 Notice that for $\gamma_++\gamma_-$, there are two ways to apply this inequality and we will choose the appropriate way depending on the relation between the $m_\pm^{(i)}$'s. More precisely, if 
\begin{enumerate}
\item $M_+=m_+^{(i)}$ and $M_-=m_-^{(i)}$, then $\alpha_1=\gamma_+^M$ and $\alpha_2=\gamma_-^M$;
\item $M_+=m_+^{(i)}$ and $M_-=m_-^{(i')}$, $i\neq i'$, then $\alpha_1=\gamma_+^M$ and $\alpha_2=\gamma_-^m$.
\end{enumerate}
To compute $(\alpha_1+\alpha_2)$ and $(\beta_1+\beta_2)$, one case, namely when $\alpha_1=\gamma_+^M$ and $\alpha_2=\gamma_-^M$ has already appeared in (\ref{e:case-1}). The other possibilities are of the form:
\begin{enumerate}
\item  $(2a+1,2a-1,2a-1,\dots,3,3,1,1)+(2b-1,2b-1,2b-3,2b-3,\dots,1,1)$
\begin{align}\label{e:case-2}
   =\begin{cases}
    (2(a+b),2(a+b)-2,2(a+b)-4,\dots,2(a-b)+2,2(a-b)+1,&\\\ \ \ \ \ \ \ \ \ \  2(a-b)-1,2(a-b)-1,\dots,3,3,1,1),& \ \ \ a\ge b,\\
    (2(a+b),2(a+b)-2,2(a+b)-4,\dots,2(b-a),2(b-a)-1,&\\\ \ \ \ \ \ \ \ 2(b-a)-3,2(b-a)-3,\dots,3,3,1,1),& \ \ \ a<b;
\end{cases}
    \end{align}
    
    \item $(2a,2a-2,2a-2,2a-4,2a-4,\dots,2,2)+(2b,2b,2b-2,2b-2,\dots,2,2)$
   \begin{align}\label{e:case-3}
   =\begin{cases}
    (2(a+b),2(a+b)-2,2(a+b)-4,\dots,2(a-b),&\\\ \ \ \ \ \ \ \ \ \  2(a-b)-2,2(a-b)-2,\dots,4,4,2,2),& \ \ \ a\ge b+1,\\
    (2(a+b),2(a+b)-2,2(a+b)-4,\dots,2(b-a)+2,&\\\ \ \ \ \ \ \ \ 2(b-a),2(b-a),\dots,4,4,2,2),& \ \ \ a\le b;
\end{cases}
    \end{align} 
    
    \item $(2a,2a,2a-2,2a-2,\dots,2,2)+(2b-1,2b-1,2b-3,2b-3,\dots,1,1)$
    \begin{align}\label{e:case-4}
   =\begin{cases}
    (2(a+b)-1,2(a+b)-1,2(a+b)-5,2(a+b)-5,\dots,2(a-b)+3, 2(a-b)+3,&\\\ \ \ \ \ \ \ \ \ \  2(a-b),2(a-b),\dots,4,4,2,2),& \ \ \ a\ge b,\\
    (2(a+b)-1,2(a+b)-1,2(a+b)-5,2(a+b)-5\dots,2(b-a)+3,2(b-a)+3&\\\ \ \ \ \ \ \ \ 2(b-a)-1,2(b-a)-1,\dots,3,3,1,1),& \ \ \ a<b;
\end{cases}
\end{align}
\end{enumerate}

Suppose $M_+=m_+^{(1)}$ and $M_-=m_-^{(1)}$. Then 
\[\gamma_++\gamma_-\ge (\gamma_+^M+\gamma_-^M)\sqcup (\gamma_+^m+\gamma_-^m)
\]
The right hand side equals the union of (\ref{e:case-1}), in which $a=m_+^{(1)}$ and $b=m_-^{(1)}$, and (\ref{e:case-4}), in which $a=m_+^{(2)}$ and $b=m_-^{(2)}$. It is not difficult, albeit a bit tedious, to compute the $C$-collapse and find that
\[
((\gamma_+^M+\gamma_-^M)\sqcup (\gamma_+^m+\gamma_-^m))^{-}_C
\]
equals exactly the disjoint union of the four partitions in the description of ${}^{\bar{\mathsf k}}\WF(\pi)$. 

Notice that one could be tempted to use the inequality
\[
((\gamma_+^M+\gamma_-^M)\sqcup (\gamma_+^m+\gamma_-^m))^{-}_C\ge (\gamma_+^M+\gamma_-^M)^{-}_C\sqcup (\gamma_+^m+\gamma_-^m),
\]
noting that  in this case $(\gamma_+^m+\gamma_-^m)$ is already a partition of type $C$. Moreover, $(\gamma_+^M+\gamma_-^M)^{-}_C$ equals, by (\ref{e:case-1}), exactly the union of two partitions in the wavefront set of the first factor ($i=1$). But $\gamma_+^m+\gamma_-^m$ is smaller than the union of the two partitions in the wavefront set of the second factor ($i=2$). So the $C$-collapse can't be separated in this computation. 

\begin{example} If $m_+^{(1)}=8$, $m_-^{(1)}=5$, $m_+^{(2)}=6$, $m_-^{(2)}=2$, then $M_+=8$, $M_-=5$, $m_+=6$, $m_-=3$, and 
\begin{align*}
\gamma_+^M+\gamma_-^M&=(27,23,23,19,19,15,15,11,11,7,7,5,5,3,3,1,1)\\
\gamma_-^m+\gamma_-^m&=(17,17,13,13,9,9,6,6,4,4,2,2),
\end{align*}
so $(\gamma_+^M+\gamma_-^M)\sqcup (\gamma_-^m+\gamma_-^m)$ equals
\[
(27,23,23,19,19,17,17,15,15,13,13,11,11,9,9,7,7,6,6,5,5,4,4,3,3,2,2,1,1).
\]
Removing a $1$ and applying the $C$-collapse, we get
\begin{align*}
(26,&24,22,20,18,18,16,16,14,14,12,12,10,10,8,8,6,6,6,6,4,4,4,4,2,2,2,2)\\
&=[26]\sqcup [18]\sqcup [6]\sqcup [6]\\
&=[2(m_+^{(1)}+m_-^{(1)})]\sqcup [2(m_+^{(2)}+m_-^{(2)})]\sqcup [2(m_+^{(1)}-m_-^{(1)})]\sqcup [2(m_+^{(2)}-m_-^{(2)})],
\end{align*}
as required.
\end{example}

Suppose $M_+=m_+^{(1)}$ and $M_-=m_-^{(2)}$. Then 
\[\gamma_++\gamma_-\ge (\gamma_+^M+\gamma_-^m)\sqcup (\gamma_+^m+\gamma_-^M)
\]
The right hand side equals the union of (\ref{e:case-2}), in which $a=m_+^{(1)}$ and $b=m_-^{(1)}$, and (\ref{e:case-3}), in which $a=m_-^{(2)}$ and $b=m_+^{(2)}$. Again we compute the $C$-collapse and find that
\[
((\gamma_+^M+\gamma_-^M)\sqcup (\gamma_+^m+\gamma_-^m))^{-}_C
\]
equals the disjoint union of the four partitions in the description of ${}^{\bar{\mathsf k}}\WF(\pi)$. This completes the proof.

\begin{example} If $m_+^{(1)}=8$, $m_-^{(1)}=3$, $m_+^{(2)}=6$, $m_-^{(2)}=5$, then again $M_+=8$, $M_-=5$, $m_+=6$, $m_-=3$, but
\begin{align*}
\gamma_+^M+\gamma_-^m&=(22,20,18,16,14,12,11,9,9,7,7,5,5,3,3,1,1)\\
\gamma_-^M+\gamma_+^m&=(22,20,18,16,14,12,10,8,6,4,2,2),
\end{align*}
so $(\gamma_+^M+\gamma_-^M)\sqcup (\gamma_-^m+\gamma_-^m)$ equals
\[
(22,22,20,20,18,18,16,16,14,14,12,12,11,10,9,9,8,7,7,6,5,5,4,3,3,2,2,1,1).
\]
Removing a $1$ and applying the $C$-collapse, we get
\begin{align*}
(22,&22,20,20,18,18,16,16,14,14,12,12,10,10,10,8,8,8,6,6,6,4,4,4,2,2,2,2)\\
&=[22]\sqcup [22]\sqcup [10]\sqcup [2]\\
&=[2(m_+^{(1)}+m_-^{(1)})]\sqcup [2(m_+^{(2)}+m_-^{(2)})]\sqcup [2(m_+^{(1)}-m_-^{(1)})]\sqcup [2(m_+^{(2)}-m_-^{(2)})],
\end{align*}
as required.
\end{example}

\end{proof}

 We can now consider the general case by putting together the two cases in Propositions \ref{p:no-unitary} and \ref{p:unitary}.

\begin{theorem}\label{t:symplectic}
   Conjecture \ref{conj-cusp} holds in the case when $\pi$ be a depth-zero supercuspidal $Sp(2n,\mathsf k)$-representation.
\end{theorem}

\begin{proof}Let $\lambda_{\pm}^\vee$ be the orthogonal partition for the dual nilpotent orbit given by the Lust--Stevens algorithm for the polynomials $P$ with degree equal to $1$. That is, $\lambda_{\pm}^\vee$ has an expression like in Proposition \ref{p:no-unitary}(1). Let $\mu_{\pm}$ be the symplectic partition for the wavefront set again for the $\deg P=1$ part as in Proposition \ref{p:no-unitary}(2). 

Similarly, let $\lambda^\vee_{\text{uni}}$ and $\mu_{\text{uni}}$ be the orthogonal (with an entry $1$ removed) and symplectic partitions that encode the dual nilpotent orbit and the wavefront set, respectively, in the case when all polynomials have degree strictly greater than $1$, i.e., in the setting of Proposition \ref{p:unitary}(1),(2).

Then 
\[
\bO^\vee_\pi=\lambda^\vee_{\pm}\sqcup \lambda^\vee_{\text{uni}}\text{ and }
{}^{\bar{\mathsf k}}\WF(\pi)=\mu_{\pm}+\mu_{\text{uni}}.
\]
Then
\begin{align}\label{e:combine}
d^\vee(\bO^\vee_\pi)&=d^\vee(\lambda^\vee_{\pm})+(\lambda^\vee_{\text{uni}})^t\ge \mu_{\pm}+\mu_{\text{uni}},
\end{align}
since by Propositions \ref{p:no-unitary} and \ref{p:unitary}, we know that
\[
d^\vee(\lambda^\vee_{\pm})\ge \mu_\pm\text{ and } (\lambda^\vee_{\text{uni}})^t\ge \mu_{\text{uni}}.
\]
In the case when $n_2=0$ (hyperspecial), equality holds for both, therefore for the combination.    

To justify the first equality in (\ref{e:combine}), write $\lambda^\vee_{\text{uni}}=\underline p_A\sqcup \underline p_A$, where $\underline p_A=\bigsqcup_P[m_P^{(1)}+m_P^{(2)}]^{k_P}\sqcup [|m_P^{(1)}-m_P^{(2)}|-1]^{k_P}$ in the notation of Proposition \ref{p:unitary}. Then we may regard $\mathbb O_\pi^\vee$ as the $G^\vee$-saturation of an orbit $\mathbb O^\vee_0$ on a maximal Levi subgroup of type $GL(n-m)\times SO(2m+1)$, where the nilpotent $GL$-part of $\mathbb O_0^\vee$ is given by $\underline p_A$ and the $SO$-part by $\lambda_\pm^\vee$. Using Lemma \ref{l:prop-nil}(3), it follows that
\begin{equation}
\begin{aligned}
d^\vee(\mathbb O^\vee_\pi)&=\operatorname{ind}_{GL(n-m)\times Sp(2m)}^{Sp(2n)}(\underline p_A^t,d^\vee(\lambda^\vee_\pm))=(\underline p_A^t+\underline p_A^t+d^\vee(\lambda^\vee_\pm)_C\\&=((\lambda_{\text{uni}}^\vee)^t+d^\vee(\lambda^\vee_\pm))_C=(\lambda_{\text{uni}}^\vee)^t+d^\vee(\lambda^\vee_\pm),
\end{aligned}
\end{equation}
where the last equality follows because all the parts in both partitions are even.
\end{proof}

\subsection{Unramified unitary groups} In this case, the calculation is simply that from Proposition \ref{p:unitary}, so there is nothing more to check.

\subsection{Special orthogonal groups} The calculation is similar to that for the symplectic group. The part of the calculation pertaining to the polynomials $P$ of degree strictly greater than $1$ is identical to Proposition \ref{p:unitary}. Thus, we only need to prove the analogue of Proposition \ref{p:no-unitary} for orthogonal groups.

\begin{proposition}\label{p:so} Let $\pi$ be a depth-zero supercuspidal $SO(m,\mathsf k)$-representation.    Suppose there are no endoscopic unitary factors, i.e., no polynomials $P$ with $\deg P>1$ appear. Then,
    \begin{enumerate}
    \item the nilpotent Langlands parameter $\bO^\vee_\pi$ is given by the strings (2o) and (3o) in section \ref{s:nil-L}.
    \item \[{}^{\bar{\mathsf k}}\WF(\pi)=\WF_1\sqcup \WF_2,\]
    where
    \[
    \WF_i=\begin{cases} [2(m^{(i)}_++m^{(i)}_-)+1]\sqcup [2|m^{(i)}_+-m^{(i)}_-|-1], & n_i^{an}=1,\\
   [2(m^{(i)}_++m^{(i)}_-)-1]\sqcup [2|m^{(i)}_+-m^{(i)}_-|-1] &n_i^{an}=0,2.
    \end{cases}
    \]
    \item Conjecture \ref{conj-cusp} holds in this case.
    \end{enumerate}
\end{proposition}

\begin{proof}
    Claim (1) is only recorded here for convenience. For (2), we need to compute the Kawanaka wavefront set of the cuspidal representations on each factor of $\CG^\circ_{n_1,n_2}$.  Consider one factor $SO(n+n^{an},n,\bF_q)$. We only need to compute the orthogonal nilpotent orbit that is attached via Springer's correspondence to the $j$-induced representation
    \[
    \sigma=\begin{cases} j_{W(B)\times W(B)}^{W(B)}\left(((m_+,\dots,2,1)\times (m_+,\dots,2,1))\boxtimes (m_-,\dots,2,1)\times (m_-,\dots,2,1))\right),& n^{an}=1,\\
    j_{W(D)\times W(D)}^{W(D)}\left(((m_+,\dots,2,1)\times (m_+-1,\dots,2,1))\boxtimes (m_-,\dots,2,1)\times (m_--1,\dots,2,1))\right),& n^{an}=0,2.
    \end{cases}
    \]
    As in the symplectic case, by the formulae in \cite{lusztig-j-ind}, this equals
    \[
    \sigma=\begin{cases}
        ((m_+,\dots,2,1)+(m_-,\dots,2,1))\times ((m_+,\dots,2,1)+(m_-,\dots,2,1)),& n^{an}=1,\\
    ((m_+,\dots,2,1)+(m_-,\dots,2,1))\times ((m_+-1,\dots,2,1)+(m_--1,\dots,2,1)),& n^{an}=0,2.
    \end{cases}
    \]
    Applying the combinatorial algorithm for the Springer representation for {\it type $B$},
we find that the corresponding {\it orthogonal} nilpotent orbit is
\[
\begin{cases} [2(m_++m_-)+1]\sqcup [2|m_+-m_-|-1], & n^{an}=1,\\
[2(m_++m_-)-1]\sqcup [2|m_+-m_-|-1],& n^{an}=0,2.
\end{cases}
\]
Then (2) follows. Claim (3) is verified by analogous computations to the symplectic case.

\end{proof}

By combining Proposition \ref{p:so} with Proposition \ref{p:unitary}, it follows, just as in the proof of Theorem \ref{t:symplectic} that Conjecture \ref{conj-cusp} holds for all depth-zero supercuspidals representations of a special orthogonal group.

\subsection{Ramified unitary groups} For the ramified unitary groups, the finite group $\CG^\circ_{n_1,n_2}(\bF_q)$  on which the cuspidal representation lives is a product of a symplectic and special orthogonal group. For notational purposes, assume $\CG^\circ_{n_1,n_2}(\bF_q)=Sp(2n_1,\bF_q)\times SO(n_2+n_2^{an},n_2,\bF_q)$. We only record the analogue of Propositions \ref{p:no-unitary} and \ref{p:so}. The proof is simply a combination of the symplectic and orthogonal cases discussed before.

\begin{proposition}\label{p:ramified} Let $\pi$ be a depth-zero supercuspidal representation of a ramified unitary group induced from $\CG^\circ_{n_1,n_2}$ as above.    Suppose there are no endoscopic unitary factors, i.e., no polynomials $P$ with $\deg P>1$ appear. Then, 
    \begin{enumerate}
    \item the nilpotent Langlands parameter $\bO^\vee_\pi$ is given by the strings (2r) and (3r) in section \ref{s:nil-L}.
    \item \[{}^{\bar{\mathsf k}}\WF(\pi)=\WF_1\sqcup \WF_2,\]
    where
    \[
    \WF_1=[2(m_+^{(1)}+m_-^{(1)})]\sqcup \begin{cases}[2(m_+^{(1)}-m_-^{(1)})],&m_+^{(1)}\ge m_-^{(1)},\\
[2(m_-^{(1)}-m_+^{(1)}-1)],&m_+^{(1)}< m_-^{(1)}
\end{cases}
    \]
   and
    \[
    \WF_2=\begin{cases} [2(m^{(2)}_++m^{(2)}_-)+1]\sqcup [2|m^{(2)}_+-m^{(2)}_-|-1], & n_2^{an}=1,\\
   [2(m^{(2)}_++m^{(2)}_-)-1]\sqcup [2|m^{(2)}_+-m^{(2)}_-|-1] &n_2^{an}=0,2.
    \end{cases}
    \]
    \item Conjecture \ref{conj-cusp} holds in this case.
    \end{enumerate}
    
\end{proposition}

\section{\bf Tempered $L$-packets of ${\bfG_2}$}\label{s:G2} 

Langlands correspondences for $G_2$ were recently constructed in \cite{AX22} and \cite{GS1,GS23} when the characteristic of $\sk$ is zero. We expect that their constructions coincide. 

Here, we prove Conjecture \ref{conj-main-2}
using the explicit description of $L$-packets in \cite{GS1, GS23} when $\sk$ has characteristic 0.  We also compute the unramified wavefront sets of depth zero tempered representations of $G_2$ using the description of Langlands parameters in \cite{AX22}. 

\subsection{The verification of the main conjectures for $G_2$} Using the explicit correspondence and the results on the wavefront sets of $G_2$ from \cite{JLS} and \cite{LoSa}, we can verify our main conjectures. Recall that $G_2$ has five geometric orbits labelled in the Bala-Carter classification, as follows (in the closure order):
\[
\Ureg,\quad \Usreg, \quad \Usr,\quad \Ulr, \quad \Utr.
\]

\begin{lemma}\label{lemma: wf subreg} If $\pi\in\Irr\,G_2$ is a nontrivial non-generic, then ${}^{\bar{\mathsf k}}\WF(\pi)=\Usreg.$ 
\end{lemma}
\begin{proof}
By \cite[\S12 and Table $G_2$, p.~437]{JLS}, the only possible geometric wavefront sets in $G_2$ are $\Ureg$, $\Usreg$, $\Ulr$, and $\Utr$. By \cite[Theorem 1.2]{LoSa}, the minimal orbit $\Ulr$ does not occur as a wavefront set. Hence $\Ureg$, $\Usreg$, and $\Utr$ are the only possibilities. The wavefront set of $\pi$ is the regular orbit $\Ureg$ if and only if $\pi$ is generic. On the other hand, the wavefront set is $1$ if and only if $\pi$ is an unramified character, hence trivial. 
\end{proof}

\begin{theorem} Let $\pi$ be a tempered irreducible representation of $G_2$. Let $\bO_\pi^\vee\in \left\{\Ureg,\Usreg,\Usr,\Ulr,\Utr \right\}$ be the nilpotent part of the Langlands parameter of $\pi$.
\begin{enumerate}
\item Suppose $\pi\ne\St_{G_2}$. If $\bO_\pi^\vee$ is not the trivial orbit, $\AZ(\pi)$ is not generic and $^{\bar\sk}\WF(\AZ(\pi))=\Usreg.$
\item If $\pi=\St_{G_2}$, we have $\bO^\vee_\pi=\Ureg$, $\AZ(\pi)=1_{G_2}$ and $^{\bar\sk}\WF(\AZ(\pi))=1$.
\item If $\bO^\vee_\pi$ is trivial, $^{\bar\sk}\WF(\AZ(\pi))$ is either $\Ureg$ or $\Usreg$.
\end{enumerate}
\end{theorem}

\begin{proof} (2) and (3) are straightforward.

Now suppose $\pi\ne\St_{G_2}$ and $\bO_\pi^\vee\neq1$. We fix a maximal $\sk$-split torus $T$, a Borel subgroup $B$ with $T\subset B$. Let $\Delta= \left\{\alpha,\beta \right\}$ be the set of simple roots associated to $B$ where $\alpha$ is a short root.  Let $P_\gamma=M_\gamma U_\gamma$, $\gamma\in\Delta\cup \left\{\emptyset \right\}$ be the parabolic subgroup associated to $\gamma$. 
We identify $T$ with $\sk^\times\times \sk^\times$ via $M_\alpha\simeq GL_2$. We also adopt the notation from \cite{GS1}:

\begin{itemize}
\item $I(s_1,s_2,\chi_1,\chi_2)=i_B^G(\nu_\sk^{s_1}\chi_1\otimes\nu_\sk^{s_2}\chi_2)$
where $\chi_i$ is a character of $\sk^\times$ and $\nu_\sk=|\cdot|_\sk.$ 
\item $I_{P_\gamma}(s,\tau)=i_{P_\gamma}^G(\nu^s\otimes\tau)$ where $\nu=\nu_\sk\circ\det$ and $\tau$ is a representation of $M_\gamma\simeq GL_2$.
\item $J_R(\cdots)$, where $R$ is a parabolic subgroup, denotes a Langlands quotient of $I_R(\cdots)$.
\item $G:=G_2$, $P:=P_\alpha$, $Q:=P_\beta$
\end{itemize}

To calculate $\AZ$ duals explicitly, we use various calculations in the classification of unitary dual of $G_2$ by Muic (\cite{Muic}). These are also summarized in \cite[Propositions 3.1, 3.2]{GS1}. We also need the following facts for the calculation:

\begin{itemize}
\item [(i)] The $\AZ$ dual of an irreducible representation $\pi$ is irreducible.
\item [(ii)] The $\AZ$ dual map is compatible with parabolic inductions, that is, $i_{P}^G\AZ(\tau)=\AZ(i_{P}^G\tau)$, where $\tau$ is an admissible representation of $M$.
\item [(iii)] If $\pi$ is irreducible, $\pi$ and $\AZ(\pi)$ have the same supercuspidal support. Hence, if $\tau$ is a supercuspidal representation of $M_\gamma$, $\AZ(i_{P_\gamma}^G\tau)=i_{P_\gamma}^G\tau$ in the Grothendieck group $K(G_2)$ of smooth representations.
\item [(iv)] Lemma \ref{lemma: wf subreg}. 
\end{itemize}

\medskip

We now explain the calculation of the $\AZ$ map. Since subquotients of the parabolically induced presentations are preserved by the $\AZ$ map, we can consider the $\AZ$ map restricted to a finite set of representations with the same cuspidal support. We will show two cases since the rest of calculation is similar: In the following, $=$, $+$ are operations in $K(G_2)$.

\medskip

\noindent
{\bf Case 1.} Subquotients of $I(1,0,1_{\sk^\times},1_{\sk^\times})$: 
By \cite[Prop. 3.1-3.2]{GS1}, we have in $K(G_2)$
\[ 
\begin{array}{lcl}I(1,0,1_{\sk^\times},1_{\sk^\times})&=&I_P(\frac12,\delta(1))+I_P(\frac12,1_{GL_2})\\
&=& I_Q(\frac12, \delta(1))+I_Q(\frac12,1_{GL_2})\\
I_P(\frac12,\delta(1))
&=& \pi_{\textup{gen}}(1)+J_P(\frac12,\delta(1))+J_Q(\frac12,\delta(1))\\
I_P(\frac12,1_{GL_2})
&=& \pi_{\textup{deg}}(1)+ J_Q(1,\pi(1,1))+ J_Q(\frac12,\delta(1))\\
I_Q(\frac12, \delta(1))
&=& \pi_{\textup{gen}}(1)+\pi_{\textup{deg}}(1)+ J_Q(\frac12,\delta(1))\\
I_Q(\frac12,1_{GL_2})
&=& J_Q(\frac12,\delta(1))+J_Q(1,\pi(1,1))+J_P(\frac12,\delta(1))
\end{array}
\]
where $\delta(1)$ is the Steinberg representation of $GL_2$. Since $\AZ_{GL_2}(\delta(1))=1_{GL_2}$,
we have
\[
\begin{array}{lccl}
\AZ: & \left\{\pi_{\textup{gen}}(1),\ J_P(\frac12,\delta(1)),\ J_Q(\frac12,\delta(1))\right\}
&\Leftrightarrow &
\left\{\pi_{\textup{deg}}(1),\  J_Q(1,\pi(1,1)),\  J_Q(\frac12,\delta(1))\right\}\\
&\left\{\pi_{\textup{gen}}(1),\ \pi_{\textup{deg}}(1),\  J_Q(\frac12,\delta(1))\right\}&\Leftrightarrow&
\left\{J_Q(\frac12,\delta(1)),\ J_Q(1,\pi(1,1)),\ J_P(\frac12,\delta(1))\right\}
\end{array}
\]
Then, the only possible matching is
\[
\pi_{\textup{gen}}(1)\overset\AZ\leftrightarrow 
J_Q(1,\pi(1,1)),\quad
\pi_{\textup{deg}}(1)\overset\AZ\leftrightarrow 
J_P(1/2,\delta(1)), \quad
J_Q(1/2,\delta(1))\overset\AZ\leftrightarrow 
J_Q(1/2,\delta(1))
\]

\medskip

\noindent
{\bf Case 2.} Subquotients of $I(1,0,\chi,\chi)=I(1,0,\chi^{-1},\chi^{-1})$ with $\chi^3=1$: The above method gives only 
\[
\AZ:  \left\{\pi_{\textup{gen}}(\chi)=\pi_{\textup{gen}}(\chi^{-1}), J_P(1/2,\delta(\chi) \right\}\ \Leftrightarrow\ 
\left\{J_P(1/2,\delta(\chi^{-1}),J_Q(1/2,\pi(\chi,\chi^{-1})) \right\}
\]
However, using $I(1,0,\chi,\chi)=I(1,0,\chi^{-1},\chi^{-1})$, we can nail it down to
\[
\pi_{\textup{gen}}(\chi)\overset\AZ\leftrightarrow
J_Q(1/2,\pi(\chi,\chi^{-1})),\qquad
J_P(1/2,\delta(\chi)\overset\AZ\leftrightarrow
J_P(1/2,\delta(\chi^{-1}))
\]
{The above calculation is summarized in Table \ref{t:G2-complete}. The calculation of $\bO_\pi^\vee$ in the table is from \cite[\S3.5]{GS1}.
The proof of the theorem is now complete.
}
\end{proof}

\noindent
\begin{table}[h]
\begin{tabular}{|c|c|c|c|c|c|c|}
\hline
& $\pi$ & $\bO_\pi^\vee$ & $d^\vee(\bO_\pi^\vee)$ & $\AZ(\pi)$ & $^{\bar\sk}\WF(\AZ(\pi))$ & reference in \cite{GS1}\\
\hline\hline
\multirow{4}{*}{$I(2,1,1_{\sk^\times},1_{\sk^\times})$}&
$\St_{G_2}$ & $\Ureg$ & $\Utr$ & $1_{G_2}$& $\Utr$&
\multirow{17}{*}{\begin{tabular}{c}Prop. 3.1 (ii)-(iii)\\Prop. 3.2 (ii)-(iii)\\\end{tabular}}\\
\cline{2-6} 
& $J_Q(\frac52,\delta(1))$ & $\Usr$ & $\Usreg$ & $J_P(\frac32,\delta(1))$ & $\Usreg$ & \\
\cline{2-6} 
& $J_P(\frac32,\delta(1))$ & $\Ulr$ & $\Usreg$ & $J_Q(\frac52,\delta(1))$ & $\Usreg$ & \\
\cline{2-6} 
& $1_{G_2}$ & $\Utr$ & $\Ureg$ & $\St_{G_2}$ & $\Ureg$ & \\
\cline{1-6}
\multirow{5}{*}{$I(1,0,1_{\sk^\times},1_{\sk^\times})$}&
$\pi_{\textup{gen}}(1)$ & $\Usreg$ & $\Usreg$ & $J_Q(1,\pi(1,1))$& $\Usreg$& \\
\cline{2-6} 
&$\pi_{\textup{deg}}(1)$ & $\Usreg$ & $\Usreg$ & $J_P(\frac12,\delta(1))$& $\Usreg$& \\
\cline{2-6}
& $J_Q(\frac12,\delta(1))$ & $\Usr$ & $\Usreg$ & $J_Q(\frac12,\delta(1))$ & $\Usreg$& \\
\cline{2-6}
& $J_P(\frac12,\delta(1))$ & $\Ulr$ & $\Usreg$ & $\pi_{\textup{deg}}(1)$ & $\Usreg$& \\
\cline{2-6}
& $J_Q(1,\pi(1,1))$ & $\Utr$ & $\Ureg$ & $\pi_{\textup{gen}}(1)$ & $\Ureg$& \\
\cline{1-6} 
\multirow{4}{*}{$\begin{array}{c}I(1,0,\chi\otimes\chi)\\\chi^2=1\end{array}$}
& $\pi_{\textup{gen}}(\chi)$ & $\Usreg$ & $\Usreg$ & 
$J_Q(\frac12,\pi(1,\chi))$ & $\Usreg$ & \\
\cline{2-6}
& $J_Q(\frac12,\delta(\chi)$ & $\Usr$ & $\Usreg$ & $J_P(\frac12,\delta(\chi))$ & $\Usreg$ & \\
\cline{2-6}
& $J_P(\frac12,\delta(\chi)$ & $\Ulr$ & $\Usreg$ & $J_Q(\frac12,\delta(\chi))$ & $\Usreg$ & \\
\cline{2-6}
& $J_Q(\frac12,\pi(1,\chi))$& $\Utr$ & $\Ureg$ & $\pi_{\textup{gen}}(\chi)$ & $\Ureg$ & \\
\cline{1-6} 
\multirow{4}{*}{$\begin{array}{c}I(1,0,\chi\otimes\chi)\\\chi^3=1\end{array}$}
& $\pi_{\textup{gen}}(\chi)$ & $\Usreg$ & $\Usreg$ & 
$J_Q(\frac12,\pi(1,\chi))$ & $\Usreg$ & \\
\cline{2-6}
& $J_P(\frac12,\delta(\chi)$ & $\Ulr$ & $\Usreg$ & $J_P(\frac12,\delta(\chi^{-1}))$ & $\Usreg$ & \\
\cline{2-6}
& $J_P(\frac12,\delta(\chi^{-1})$ & $\Ulr$ & $\Usreg$ & $J_Q(\frac12,\delta(\chi))$ & $\Usreg$ & \\
\cline{2-6}
& $J_Q(1,\pi(\chi,\chi^{-1}))$& $\Utr$ & $\Ureg$ & $\pi_{\textup{gen}}(\chi)$ & $\Ureg$ & \\
\hline
$i_B^G\chi$ irred.& $i_B^G\chi$& $\Utr$ & $\Ureg$ & $i_B^G\chi$& $\Ureg$ & \\
\hline\hline
\multirow{2}{*}{$\begin{array}{c}I_P(\frac12,\tau)\\ \tau\simeq\tau^\vee,\ \omega_\tau=1\end{array}$} & $\delta_P(\tau)$ & $\Ulr$ & $\Usreg$ & $J_P(\frac12,\tau)$ & $\Usreg$ & \multirow{2}{*}{Prop. 3.1-(i)}\\
\cline{2-6}
& $J_P(\frac12,\tau)$ & $\Utr$& $\Ureg$ & $\delta_P(\tau)$ & $\Ureg$ & \\
\hline
\multirow{2}{*}{$\begin{array}{c}I_P(0,\tau), \tau\simeq\tau^\vee\\ \omega_\tau\neq1\end{array}$} & $I_P(\tau)_{\textup{gen}}$ & $\Utr$ & $\Ureg$ & $I_P(\tau)_{\textup{deg}}$ & $\Usreg$ & \multirow{2}{*}{Prop. 3.1-(ii)}\\
\cline{2-6}
& $I_P(\tau)_{\textup{deg}}$ & $\Utr$& $\Ureg$ & $I_P(\tau)_{\textup{gen}}$ & $\Ureg$ & \\
\hline
$i_P^G\tau$ irred.& $i_P^G\tau$& $\Utr$ & $\Ureg$ & $i_P^G\tau$& $\Ureg$ & \\
\hline\hline
\multirow{2}{*}{$\begin{array}{c}I_Q(\frac12,\tau)\\ \tau\simeq\tau^\vee,\ \omega_\tau=1\end{array}$} & $\delta_Q(\tau)$ & $\Usr$ & $\Usreg$ & $J_Q(\frac12,\tau)$ & $\Usreg$ & \multirow{2}{*}{Prop. 3.2-(i)}\\
\cline{2-6}
& $J_Q(\frac12,\tau)$ & $\Utr$& $\Ureg$ & $\delta_Q(\tau)$ & $\Ureg$ & \\
\hline
\multirow{2}{*}{$\begin{array}{c}I_Q(1,\tau), \ \tau\simeq\tau^\vee\\ \tau\textrm{ dihedral},\ \omega_\tau\neq1\end{array}$} & $\pi_{\textup{gen}}(\tau)$ & $\Usreg$ & $\Usreg$ & $J_Q(1,\tau)$ & $\Usreg$ & \multirow{2}{*}{Prop. 3.2-(ii)}\\
\cline{2-6}
& $J_Q(1,\tau)$ & $\Utr$& $\Ureg$ & $\pi_{\textup{gen}}(\tau)$ & $\Ureg$ & \\
\hline
\multirow{2}{*}{$\begin{array}{c}I_Q(0,\tau), \tau\simeq\tau^\vee\\ \omega_\tau\neq1\end{array}$} & $I_Q(\tau)_{\textup{gen}}$ & $\Utr$ & $\Ureg$ & $I_Q(\tau)_{\textup{deg}}$ & $\Usreg$ & \multirow{2}{*}{Prop. 3.2-(iii)}\\
\cline{2-6}
& $I_Q(\tau)_{\textup{deg}}$ & $\Utr$& $\Ureg$ & $I_Q(\tau)_{\textup{gen}}$ & $\Ureg$ & \\
\hline
$i_Q^G\tau$ irred.& $i_Q^G\tau$& $\Utr$ & $\Ureg$ & $i_Q^G\tau$& $\Ureg$ & \\
\hline\hline
 {$\pi$, unipotent s.c.} & {$\pi$} &  {$\Usreg$}
&  {$\Usreg$} &  {$\pi$} &  {$\Usreg$} &  {\S3.5-(2)}\\ 
\hline
 {$\pi$, singular s.c.} & {$\pi$} &  {$\Usr,\ \Ulr$}
&  {$\Usreg$} &  {$\pi$} &  {$\Usreg$} &  {\S3.5-(3),(4)}\\ 
\hline
 {$\pi$, nonsingular s.c.}&  {$\pi$} &  {$\Utr$} &  {$\Ureg$} &  {$\pi$} & $\Usreg$ or $\Ureg$ &  {\S3.5-(5)}\\
\hline
\end{tabular}

\medskip

\caption{Irreducible tempered representations of $G_2$ and their $\AZ$-duals}\label{t:G2-complete}
\end{table}


\begin{corollary}
Conjecture \ref{conj-main-2}-(1) holds for any irreducible tempered representation $\pi$ on $G_2$.
\end{corollary}
\begin{proof}
Conjecture \ref{conj-main-2}-(1) is immediate from Table \ref{t:G2-complete}. 
\end{proof}

\begin{corollary}
For the $L$-packets in \cite{GS1, GS23}, 
Conjecture \ref{conj-main-2}-(2) holds. 
\end{corollary}

\begin{proof}
Write $\varphi:WD_\sk\longrightarrow {}^L\!G$ for a Langlands parameter where $WD_\sk=W_\sk\times SL_2(\bC)$. Conjecture \ref{conj-main-2}-(2) results from case by case checking summarized in Table \ref{t:G2-L-packet}. In the following table, $\bO_\varphi^\vee$ and $\Pi_\varphi$ denote the nilpotent orbit and the $L$-packet associated to $\varphi$ respectively. 
\begin{table}[h]
\begin{tabular}{|c|c|c|c|c|c|c|}
\hline
\multirow{2}{*}{$\bO_\varphi$} &  \multirow{2}{*}{$d^\vee(\bO_\varphi^\vee)$} & \multirow{2}{*}{$\varphi(W_\sk)$} & \multirow{2}{*}{$\Pi_\varphi$}  & \multirow{2}{*}{$\AZ(\Pi_\varphi)$} & $\pi$: $^{\bar\sk}\WF(\AZ(\pi))$ & {reference on $\varphi$}\\&&&&& $=d^\vee(\bO_\varphi^\vee)$ & in \cite{GS1}\\
\hline\hline
$\Ureg$ & 1 & 1 & $ \left\{\St_{G_2} \right\}$ & $1_{G_2}$ & $\Pi_\varphi$ & \S3.5-(1)\\
\hline
\multirow{7}{*}{$\Usreg$} & \multirow{7}{*}{$\Usreg$}& \multirow{2}{*}{1} 
& \multirow{2}{*}{$ \left\{\pi_{\textup{gen}}(1),\pi_{\textup{deg}}(1),\pi(1) \right\}$} & \multirow{2}{*}{$\begin{array}{l}\{\ J_Q(1,\pi(1,1))\\ \quad J_P(\frac12,\delta(1)), \pi(1)\ \}\end{array}$}  &\multirow{2}{*}{$\Pi_\varphi$} & \multirow{7}{*}{\S3.5-(2)}\\ &&&&&& \\
\cline{3-6}
&& \multirow{2}{*}{$\mu_2$} & \multirow{2}{*}{$\begin{array}{c}\{
\ \pi_{\textup{gen}}(\chi)\textrm{ with }\chi^2=1,\\\quad\pi(-1) \textrm{ unipotent s.c.}\ \}\end{array}$}& \multirow{2}{*}{$ \left\{J_Q(\frac12,\pi(1,\chi)),\pi(-1) \right\}$} & \multirow{2}{*}{$\Pi_\varphi$} &\\ &&&&&&\\
\cline{3-6}
&& \multirow{2}{*}{$\mu_3$} & \multirow{2}{*}{$\begin{array}{l}\{\ \pi_{\textup{gen}}(\chi)\textrm{ with }\chi^3=1,\\\pi(\omega),\pi(\omega^2): \textrm{ unip. s.c.}\ \}\end{array}$}& \multirow{2}{*}{$\begin{array}{l}\{\ J_Q(\frac12,\pi(1,\chi)),\\\quad \pi(\omega),\pi(\omega^2)\ \}\end{array}$} & \multirow{2}{*}{$\Pi_\varphi$} &\\ &&&&&&\\
\cline{3-6}
&& $S_3$ & $\pi_{\textup{gen}}(\tau)$ & $J_Q(1,\tau)$ & {$\Pi_\varphi$} &\\
\hline
\multirow{2}{*}{$\Usr$} & \multirow{2}{*}{$\Usreg$} &  {}
&  {$ \left\{\delta_P(\tau),\pi(-1) \right\}$} &  \multirow{2}{*}{$ \left\{J_P(\frac12,\tau),\pi(-1) \right\}$} &  \multirow{2}{*}{$\Pi_\varphi$} &  \multirow{2}{*}{\S3.5-(3)}\\ &&& $\pi(-1)$ singular s.c. &&&\\
\hline
 \multirow{2}{*}{$\Ulr$} & \multirow{2}{*}{$\Usreg$} &  {}
&  {$ \left\{\delta_Q(\tau),\pi(-1)  \right\}$} &  \multirow{2}{*}{$ \left\{J_P(\frac12,\tau),\pi(-1) \right\}$}  &  \multirow{2}{*}{$\Pi_\varphi$} &  \multirow{2}{*}{\S3.5-(4)}\\ &&& $\pi(-1)$ singular s.c. &&&\\
\hline
\multirow{3}{*}{$1$} &\multirow{3}{*}{$\Ureg$} & \multirow{3}{*}{}& {$\left\{i_B^G\chi, \textrm{ irred}\right\}$},  & \multirow{3}{*}{$\Pi_\varphi$} & \multirow{3}{*}{$\begin{array}{l}\textrm{generic}\\\textrm{ element}\\\ \textrm{\ in }\Pi_\varphi\end{array}$} &{Prop 3.3}\\ 
\cline{4-4}\cline{7-7}
&&& {$ \left\{i_P^G\tau \right\}$ or $ \left\{i_Q^G\tau \right\}$}  && &  {\S3.5, p17} \\
\cline{4-4}\cline{7-7}
&&&s.c. $L$-pacekt& &&{\S3.5-(5)}\\
\hline
\end{tabular}

\medskip

\caption{Tempered $L$-packets of $G_2$}\label{t:G2-L-packet}
\end{table}

\end{proof}

\subsection{Unramified wavefront sets of depth-zero representations}
In the case of the depth-zero representations, it is instructive to compute the unramified wavefront set as well in order to test several possible refinements of our main conjecture. 

\begin{proposition}
The canonical unramified wavefront sets of the depth-zero supercuspidal representations in the parametrization of \cite{AX22} are given in Table \ref{t:G2}. 
\end{proposition}

\begin{proof}
The depth-zero supercuspidals are parameterized in {\it loc. cit.} as in the following Table \ref{t:G2}. From the realization of the representations by compact induction, we can immediately compute the canonical unramified ${}^K\underline\WF(\pi)$ (in the sense of Okada \cite{okada-WF}) and the geometric wavefront sets, via Theorem \ref{t:supercuspidal-WF}. Then one verifies by inspection the conjecture. The unipotent supercuspidal representations have already been treated in \cite{CMBO-cusp}. We recall that
\[
d^\vee(\Ureg)=1,\quad d^\vee(1)=\Ureg,\text{ and } d^\vee(\bO^\vee)=\Usreg,
\]
if $\bO^\vee=\Usreg$, $\Usr$, or $\Ulr$.

We exemplify the calculation of the wavefront set for $\pi$ in the sixth row, for which $\bO^\vee=A_1$. 
Here ${\fg^\vee}^{\varphi(I_\sk)}=sl(3,\bC)$. The supercuspidal representation is compactly induced from the hyperspecial parahoric from the inflation of a cuspidal representation $\mu$ of $G_2(\bF_q)$. The cuspidal representation $\mu$ belongs to the Lusztig series corresponding to a semisimple element $\tau\in G_2^\vee(\bF_q)$, such that the centralizer of $\tau$ is $SU(3)$. (Recall that we need $q\equiv -1 (3).$) In $SU(3)$, the unipotent class parametrizing $\tau$ is $(12)$,  so to find the Kawanaka wavefront set of $\mu$ we need  the $j$-induced
\[
j_{S_3}^{W(G_2)}((12))=2_1,
\]
the reflection representation of $W(G_2)$. The corresponding nilpotent orbit via Springer's correspondence is $\Usreg.$ This is the (geometric) Kawanaka wavefront set of $\mu$. Since the parahoric is the hyperspecial, it follows that the canonical unramified wavefront set of $\pi$ is indeed $(\Usreg,1)$.

\begin{table}[h]
\begin{tabular}{|c|c|c|c|c|c|c|c|}
\hline
{$\bO^\vee$} & $\varphi(I_\sk)$ & $\varphi(\langle\Fr\rangle)$ & $G_{x,0}/G_{x,0^+}$ & $\pi$ & ${}^K\underline\WF(\pi)$ &$d^\vee(\bO^\vee)$ &$d_S({}^{K}\underline\WF(\pi))$\\
\hline
\hline
$\Usreg$ &$1$ &$1$ &$G_2$ &$G_2(1)$ &$(\Usreg,1)$ &$\Usreg$ &$\Usreg$\\
                    &       &$\mu_2$  &   &$G_2(-1)$ & & &\\
                    &       &$\mu_3$  &   &$G_2(\omega), G_2(\omega^2)$ & & &\\
\hline
$\Usreg$ &$\mu_2$ & &$SO_4$ &singular &$(\Usreg,(12))$ &$\Usreg$ &$\widetilde A_1$\\
\hline
$\Usreg$ &$\mu_3$ & &$SL_3$ &singular, $q\equiv 1 (3)$ &$(\Usreg,(123))$ &$\Usreg$ &$A_1$\\
\hline\hline
$A_1$ &$\mu_3$ && $G_2$ &singular from $SU(3)$ &$(\Usreg,1)$ &$\Usreg$ &$\Usreg$\\
&&&&$q\equiv -1(3)$ &&&\\
\hline\hline
$1$ & & &$G_2$ &nonsingular &$(\Ureg,1)$ &$\Ureg$ &$1$\\
       & & &$SL_3$ &toral &$(\Usreg,(123))$ &$\Ureg$ &$A_1$\\
       & & &$SO_4$ & &$(\Usreg,(12))$ &$\Ureg$ &$\widetilde A_1$\\
       \hline\hline
\end{tabular}
\caption{Depth-zero supercuspidal representations of $G_2$}\label{t:G2}
\end{table}


\end{proof}

To account for an example of an anti-tempered unipotent $E_7$-representation for which ${}^{\bar{\mathsf k}}\WF(\pi)<d^\vee(\bO^\vee_\pi)$ (strict inequality), Okada proposed another, more precise form of the relation between the wavefront set of unipotent representations and the Langlands parameter {\cite[Conjecture 1.4.3]{CMBO23}:
 \begin{equation}\label{e:okada-variant}
d_S({}^K\underline\WF(\pi))=G^\vee\cdot \CO^\vee_{\AZ}.
\end{equation}
}
This holds for unipotent representations with real infinitesimal character \cite{CMBO23}, and there are no known counterexamples for unipotent representations. However, for more general representations, this equality can't always hold as stated, for example, because of the case of nongeneric regular depth-zero supercuspidals, see section \ref{s:DBR}. 

A natural attempt to fix it is to intersect with $\fh^\vee={\fg^\vee}^{\varphi(I_\sk)}$, i.e., to expect that  $\overline{d_S({}^K\underline\WF(\pi))}\cap \fh^\vee$ is an irreducible $H^\vee(s)$-variety, $s=\varphi(\Fr)$, hence the closure of a nilpotent $H^\vee(s)$-orbit $\CO^\vee$ such that
{\begin{equation}
\bO^\vee=G^\vee\cdot \CO^\vee_\AZ.
\end{equation}
}
Often this seems to be the case (for example, this resolves trivially the case of regular depth-zero supercuspidals), but again not always, as the $G_2$ examples above show. 

More precisely, consider $\pi$ in the sixth row of Table \ref{t:G2} with $\bO^\vee_\pi=A_1$ and ${}^{K}\underline\WF(\pi)=(\Usreg,1)$
and $d_S({}^{K}\underline\WF(\pi))=\Usreg$. The intersection $\overline{d_S({}^K\underline\WF(\pi))}\cap \fh^\vee$ is the principal orbit in $\mathfrak{sl}(3,\bC)$. However, the Langlands nilpotent orbit $\bO^\vee=A_1$ comes from the {\it subregular} orbit in $\mathfrak{sl}(3,\bC)$.

\

Interestingly, the $G_2$ depth-zero supercuspidals in Table \ref{t:G2} seem to indicate that there is no generalization of (\ref{e:okada-variant}) even if one would replace the equality with an inequality, i.e.:
\begin{align*}
d_S({}^K\underline\WF(\pi))<\bO^\vee,\quad &\pi\text{ in the fourth or fifth rows}, \\
d_S({}^K\underline\WF(\pi))>\bO^\vee,\quad &\pi\text{ in the sixth row}. 
\end{align*}

\section{Acknowledgements} This paper builds on the previous work of the first-named author with Lucas Mason-Brown and Emile Okada. It is a pleasure to thank them for many useful conversations about the wavefront set. The second-named author thanks George Lusztig and Cheng-Chiang Tsai for insightful discussions. We also thank Baiying Liu and Freydoon Shahidi for sharing their work with us and for an interesting correspondence, and Hiraku Atobe and Alberto Minguez for clarifying the connection between Remark \ref{r:extreme} and their AZ-algorithm \cite{AM}. Special thanks are due to the referee for a very detailed and useful report that led to many improvements.

D.C. was partially supported by the EPSRC grant ``New Horizons'' EP/V046713/1, 1/2021-12/2023.

For the purpose of open access, the authors have applied a CC BY public copyright license to any author accepted manuscript arising from this submission.

\smallskip

On behalf of all authors, the corresponding author states that there is no conflict of interest.

\ifx\undefined\bysame
\newcommand{\bysame}{\leavevmode\hbox to3em{\hrulefill}\,}
\fi


\begin{thebibliography}{MMMM}


\bibitem[AM]{AM} H.~Atobe, A.~Minguez, \emph{The explicit Zelevinsky-Aubert duality}, Compositio Math. {\bf 159} (2023), no. 2, 380--418.

\bibitem[Au]{Au-inv} A.-M. Aubert, \emph{Dualit\' e dans le groupe de Grothendieck de la cat\' egorie des repr\' esentations lisses de longueur finie d'un groupe r\' eductif p-adique}, {Trans. Amer. Math. Soc.}, {\bf 347}(1995), no. 6,  2179--2189.


\bibitem[AX]{AX22} A.-M. Aubert, Y. Xu, \emph{The explicit local Langlands correspondence for $G_2$}, preprint 2022, \texttt{arXiv:2208.12391}.

 \bibitem[BM]{BM}
 D.~Barbasch, A.~Moy, \emph{Local character expansions},
   {Ann. Sci. \'{E}cole Norm. Sup. (4)},
   {\bf 30} (1997), no. 5, 553--567.

\bibitem[BV]{BV} D. Barbasch, D.A. Vogan, Jr.,
\emph{Unipotent representations of complex semisimple groups}, Ann. Math. {\bf 121} (1985), issue 1, 41--110.

\bibitem[Car]{Car} R.W. Carter,
\emph{Finite Groups of Lie Type}, Wiley, Chichester (1985).

\bibitem[CM]{CM} D. Collingwood, W. McGovern, \emph{Nilpotent orbits in semisimple Lie algebras}, Van Nostrand–Reinhold, New York (1993).

\bibitem[CMBO1]{CMBO21} D. Ciubotaru, L. Mason-Brown, E. Okada, \emph{Wavefront sets of unipotent representations of reductive p-adic groups I}, \texttt{arXiv:2112.14354v4}, to appear in {\it Amer. Jour. Math.}

\bibitem[CMBO2]{CMBO23} D. Ciubotaru, L. Mason-Brown, E. Okada, \emph{Wavefront sets of unipotent representations of reductive p-adic groups II}, J. Reine Angew. Math. (Crelles Journal) {\bf 823} (2025), 191--253. 

\bibitem[CMBO3]{CMBO-cusp} D. Ciubotaru, L. Mason-Brown, E. Okada, \emph{The Wavefront sets of unipotent supercuspidal representations}, Algebra and Number Theory {\bf 18} (2024), no. 10, 1863--1889.

\bibitem[CMBO4]{CMBO-arthur} D. Ciubotaru, L. Mason-Brown, E. Okada, 
\emph{Some unipotent Arthur packets for reductive $p$-adic groups}, IMRN (2024), issue 9, 7502--7525.

\bibitem[DBR]{DBR} S. DeBacker, M. Reeder, \emph{Depth-zero supercuspidal $L$-packets and their stability}, Annals Math. {\bf 169} (2009), 795--901. 

\bibitem[EM]{evens-mirkovic} S. Evens, I. Mirkovi\' c, \emph{Fourier transform and the Iwahori-Matsumoto involution}, {Duke Math. J.} {\bf 86}(1997), no. 5, 435--464.

\bibitem[GS1]{GS1} W.-T. Gan, G. Savin, \emph{Howe duality and dichotomy for exceptional theta correspondences}, Invent. Math., {\bf 232} (2023), no. 1, 1--78.

\bibitem[GS2]{GS23} W.-T. Gan, G. Savin, \emph{The local Langlands conjecture for $G_2$}, Forum of Mathematics, Pi, {\bf 11} (2023), 1--42.

\bibitem[HC]{HC} Harish-Chandra, \emph {Admissible invariant distributions on reductive p-adic groups. With a preface and notes by Stephen DeBacker and Paul J. Sally, Jr.}, {University Lecture Series}, {American Mathematical Society, Providence, RI}, {1999}, {\bf 16}, {xiv+97 pp.}

\bibitem[HLLS]{HLLS} A.~Hazeltine, B. Liu, C.-H.~Lo, F.~Shahidi, \emph{On the upper bound of wavefront sets of representations of $p$-adic groups}, preprint 2024, {\texttt{arXiv:2403.11976}}.

\bibitem[Ho]{Ho} R. Howe, \emph{The Fourier transform and germs of characters (case of $Gl_n$ over a $p$-adic field)}, Math. Ann. {\bf 208} (1974), 305--322.

\bibitem[Ji]{Ji} D. Jiang, \emph{Automorphic Integral transforms for classical groups I: endoscopy correspondences}, Automorphic Forms: $L$-functions and related geometry: assessing the legacy of I.I. Piatetski-Shapiro, 179--242, Comtemp. Math. {\bf 614}, 2014, AMS.


\bibitem[JLS]{JLS} D.~Jiang, B.~Liu, G.~Savin, \emph{Raising nilpotent orbits in wave-front sets}, Represent. Theory {\bf 20} (2016), 419--450.

\bibitem[Kal]{Kal} T. Kaletha, \emph{Supercuspidal $L$-packets}, 	preprint 2021, \texttt{arXiv:1912.03274v2}.

\bibitem[Kaw]{Kaw} N. Kawanaka, \emph{Shintani lifting and Gelfand-Graev representations}, Proc. Sympos. Pure Math., Vol. {\bf 47}, 147--163, Amer. Math. Soc. 1987.

\bibitem[KS]{KS} M.~Kashiwara, P. Schapira, \emph{Sheaves on Manifolds}, Grundlehren Math. Wiss. {\bf 292},
Springer-Verlag, Berlin, 1990.

\bibitem[KL]{KL} D. Kazhdan, G. Lusztig, \emph{Proof of the Deligne-Langlands conjecture for Hecke algebras}, {Invent. Math.} {\bf 1}(1987), no. 2, 153--215.

\bibitem[KV]{KV} D. Kazhdan and  Y. Varshavsky, \emph{Endoscopic decomposition of certain depth zero representations}, in ``Studies in Lie theory'', 223--301, Progr.Math. {\bf 243}, Birkh\" auser, Boston, 2006.


\bibitem[LoSa]{LoSa} H.Y.~Loke, G.~Savin, \emph{On minimal representations of Chevalley groups of type $D_n$, $E_n$, and $G_2$}, Math. Ann. {\bf 340} (2008), no. 1, 195--208.

\bibitem[LuSt]{LS} J. Lust, S. Stevens, 
\emph{On depth zero $L$-packets for classical groups}, Proc. London Math. Soc. {\bf 121} (2020), no. 5, 1083--1120.

\bibitem[Lu1]{orange} G. Lusztig, \emph{Characters of reductive groups over a finite field}, Ann.~Math. Studies {\bf 107}, Princeton Univ. Press 1984.

\bibitem[Lu2]{lusztig-unip-supp} G. Lusztig, \emph{A unipotent support for irreducible representations}, Adv. Math. {\bf 94} (1992), no. 2, 139--179.

\bibitem[Lu3]{lusztig-j-ind} G. Lusztig, \emph{Unipotent classes and special Weyl group representations}, J. Algebra {\bf 321} (2009), 3418--3449.

\bibitem[MW]{MW} C. M\oe glin, J.-L. Waldspurger, \emph{Mod\`eles de Whittaker d\'eg\'en\'er\'es pour des groupes $p$-adiques},
{Math. Z.}, {\bf 196}(1987), 427--452.

\bibitem[Mu]{Muic} G.~Mui\' c, \emph{The unitary dual of $p$-adic $G_2$}, Duke Math. J. {\bf 90} (1997), no. 3, 465--493.

\bibitem[Ok]{okada-WF} E. Okada, \emph{The wavefront set of spherical Arthur representations}, preprint 2021, 
\texttt{arXiv:2107.10591}.

\bibitem[So]{sommers-duality} E. Sommers, \emph{Lusztig's canonical quotient and generalized duality}, J. Algebra {\bf 243} (2001), issue 2, 790--812.

\bibitem[Sp]{spaltenstein} N. Spaltenstein, \emph{Classes Unipotentes et Sous-groupes de Borel}, Lecture Notes Math. {\bf 946} (1982), Springer Berlin, Heidelberg.

\bibitem[Ts]{tsai} C.C. Tsai, \emph{Geometric wave-front set may not be a singleton}, J. Amer. Math. Soc. {\bf 37} (2024), 281--304.

\bibitem[Vo]{vogan-llc} D.A. Vogan, Jr., \emph{The local Langlands conjecture}, {Representation theory of groups and algebras}, Contemp. Math., Amer. Math. Soc., Providence, RI, 1993, {\bf 145}, {305--379}.

\bibitem[Wa]{Wa} J.-L. Waldspurger, \emph{Fronts d'onde des repr\' esentations temp\' er\' ees et de r\' eduction unipotente pour 
$\mathsf{SO}(2n+1)$},
Tunisian J. Math. {\bf 2} (2020), no. 1, 43--95.

\bibitem[Ze]{zelevinsky-p-adic} A. Zelevinsky, \emph{A p-adic analogue of the Kazhdan--Lusztig hypothesis}, Functional An. Appl. {\bf 15} (1981), 83--92.

\end{thebibliography}
\end{document}